\documentclass[10pt,a4paper,twoside]{article}

\title{\sc A Posteriori Error Analysis for the Optimal Control of Magneto-Static Fields}
\def\shorttitle{On an Optimal Control Problem for Magneto-Static Fields}
\def\pauthor{Dirk Pauly \& Irwin Yousept}

\def\mylabelonoff{off}
\def\allowdisbrk{yes}

\usepackage{a4,exscale,ifthen,amsfonts,amssymb,amsmath,amscd,graphicx}
\usepackage[english]{babel}
\usepackage[mathscr]{eucal}

\author{{\sf\pauthor}}
\numberwithin{equation}{section}

\newcommand{\bewboxw}{\mbox{}\hfill $\square$ \\}

\newenvironment{proof}{{\noindent\bf Proof }}{\bewboxw}

\ifthenelse{\equal{\mylabelonoff}{on}}
{}{}
\ifthenelse{\equal{\allowdisbrk}{yes}}
{\allowdisplaybreaks}{}

\usepackage[mathscr]{eucal}
\usepackage{color}
\usepackage{graphicx}

\def\rc{\color{red}}

\newcommand{\Abb}[5]{\begin{array}{ccccc}#1&:&#2&\longrightarrow&#3\\{}&{}&#4&\longmapsto&#5\end{array}}
\newcommand{\ds}{\displaystyle}
\newcommand{\ol}{\overline}
\newcommand{\ul}{\underline}
\newcommand{\ub}{\underbrace}

\newcommand{\nz}{\mathbb{N}}

\newcommand{\rz}{\mathbb{R}}

\newcommand{\rt}{\rz^3}
\newcommand{\rttt}{\rz^{3\times3}}

\DeclareMathOperator{\p}{\partial}

\newcommand{\na}{\nabla}

\DeclareMathOperator{\rot}{rot}

\renewcommand{\div}{\operatorname{div}}

\newcommand{\eps}{\varepsilon}
\newcommand{\epsmo}{\eps^{-1}}
\newcommand{\mumo}{\mu^{-1}}
\newcommand{\muu}{\ul{\mu}}
\newcommand{\muo}{\ol{\mu}}
\newcommand{\epsu}{\ul{\eps}}
\newcommand{\epso}{\ol{\eps}}

\newcommand{\om}{\Omega}
\newcommand{\oms}{\omega}
\newcommand{\pom}{\p\!\om}
\newcommand{\poms}{\p\!\oms}

\newcommand{\ga}{\Gamma}
\newcommand{\gas}{\gamma}

\newcommand{\equi}{\Leftrightarrow}
\def\impl{\Rightarrow}

\newcommand{\qequi}{\quad\equi\quad}

\def\e{\mathrm{e}}

\DeclareMathOperator{\dist}{dist}

\newcommand{\zvec}[2]{\begin{bmatrix}#1\\#2\end{bmatrix}}
\newcommand{\dvec}[3]{\begin{bmatrix}#1\\#2\\#3\end{bmatrix}}
\newcommand{\zmat}[4]{\begin{bmatrix}#1&#2\\#3&#4\end{bmatrix}}
\newcommand{\dmat}[9]{\begin{bmatrix}#1&#2&#3\\#4&#5&#6\\#7&#8&#9\end{bmatrix}}
\DeclareMathOperator{\id}{id}

\def\set#1#2{\{#1\,:\,#2\}}

\newcommand{\ck}{c_{\mathtt{k}}}

\def\ub{\bar{u}}

\def\vt{\tilde{v}}

\def\vb{\bar{v}}

\def\Hh{\hat{H}}
\def\Hb{\bar{H}}

\def\Et{\tilde{E}}

\def\Eb{\bar{E}}

\def\et{\tilde{e}}
\def\eh{\hat{e}}
\def\eb{\bar{e}}

\def\calM{\mathcal{M}}

\def\calA{\mathcal{A}}
\def\calB{\mathcal{B}}
\def\cE{\mathcal{E}}

\DeclareMathOperator{\Lebesgue}{\mathsf{L}}
\newcommand{\Lgen}[2]{\Lebesgue^{#1}_{#2}}
\def\Li{\Lgen{\infty}{}}
\def\Liom{\Li(\om)}

\def\Lt{\Lgen{2}{}}
\def\Lteps{\Lgen{2}{\eps}}
\def\Ltmu{\Lgen{2}{\mu}}
\def\Ltom{\Lt(\om)}
\def\Ltoms{\Lt(\oms)}
\def\Ltepsom{\Lgen{2}{\eps}(\om)}

\def\Ltepsoms{\Lgen{2}{\eps}(\oms)}

\DeclareMathOperator{\Sobolev}{\mathsf{H}}
\newcommand{\Hgen}[3]{\overset{#3}{\Sobolev}{}^{#1}_{#2}}
\def\Ho{\Hgen{1}{}{}}

\def\Hk{\Hgen{k}{}{}}

\def\Hkom{\Hk(\om)}

\def\Hoc{\Hgen{1}{}{\circ}}
\def\Hkc{\Hgen{k}{}{\circ}}

\def\Hkcom{\Hkc(\om)}

\DeclareMathOperator{\Cont}{\mathsf{C}}
\newcommand{\Cgen}[2]{\overset{#2}{\Cont}{}^{#1}}

\def\Cic{\Cgen{\infty}{\circ}}

\def\Cicom{\Cic(\om)}

\newcommand{\normdst}{\hspace{-0.4ex}}
\newcommand{\scp}[2]{\left\langle#1,#2\right\rangle}
\newcommand{\scps}[2]{\langle#1,#2\rangle}
\newcommand{\scpom}[2]{\scp{#1}{#2}_{\om}}
\newcommand{\scpsom}[2]{\scps{#1}{#2}_{\om}}

\newcommand{\norm}[1]{\left|\normdst\left|#1\right|\normdst\right|}
\newcommand{\norms}[1]{|\normdst|#1|\normdst|}
\newcommand{\tnorm}[1]{\left|\normdst\left|\normdst\left|#1\right|\normdst\right|\normdst\right|}
\newcommand{\tnorms}[1]{|\normdst|\normdst|#1|\normdst|\normdst|}
\newcommand{\dnorm}[1]{\left|\normdst\left|\normdst\left|#1\right|\normdst\right|\normdst\right|}

\newcommand{\normom}[1]{\norm{#1}_{\om}}

\newtheorem{lem}{Lemma}%[section]

\newtheorem{theo}[lem]{Theorem}

\newtheorem{rem}[lem]{Remark}

\DeclareMathOperator{\rotspace}{\mathsf{R}}
\newcommand{\rotgen}[2]{\overset{#2}{\rotspace}{}_{#1}}
\newcommand{\rotgenom}[2]{\rotgen{#1}{#2}(\om)}
\newcommand{\rom}{\rotgenom{}{}}
\newcommand{\rcom}{\rotgenom{}{\circ}}
\newcommand{\rzom}{\rotgenom{0}{}}
\newcommand{\rczom}{\rotgenom{0}{\circ}}

\DeclareMathOperator{\divspace}{\mathsf{D}}
\newcommand{\divgen}[2]{\overset{#2}{\divspace}{}_{#1}}
\newcommand{\divgenom}[2]{\divgen{#1}{#2}(\om)}
\newcommand{\dom}{\divgenom{}{}}
\newcommand{\dcom}{\divgenom{}{\circ}}
\newcommand{\dzom}{\divgenom{0}{}}
\newcommand{\dczom}{\divgenom{0}{\circ}}

\newcommand{\ho}{\Ho}
\newcommand{\hoc}{\Hoc}

\newcommand{\hio}{\hi_{1}}
\newcommand{\hit}{\hi_{2}}

\DeclareMathOperator{\harmonic}{\mathcal{H}}
\newcommand{\harm}[2]{\harmonic^{#1}_{#2}}
\newcommand{\harmom}[2]{\harm{#1}{#2}(\om)}

\newcommand{\harmdepsom}{\harmom{}{\mathtt{D},\eps}}

\newcommand{\harmnmuom}{\harmom{}{\mathtt{N},\mu}}

\newcommand{\non}{\nonumber}

\newcommand{\scpsoms}[2]{\scps{#1}{#2}_{\oms}}
\newcommand{\scpsepsom}[2]{\scps{#1}{#2}_{\om,\eps}}
\newcommand{\scpsepsoms}[2]{\scps{#1}{#2}_{\oms,\eps}}

\newcommand{\scpsmuom}[2]{\scps{#1}{#2}_{\om,\mu}}

\newcommand{\scpsmumoom}[2]{\scps{#1}{#2}_{\om,\mumo}}

\newcommand{\normos}[1]{|#1|} % one bar small

\newcommand{\normosepsom}[1]{\normos{#1}_{\om,\eps}}

\newcommand{\normosepsoms}[1]{\normos{#1}_{\oms,\eps}}

\newcommand{\normosmuom}[1]{\normos{#1}_{\om,\mu}}

\newcommand{\normosmuoms}[1]{\normos{#1}_{\oms,\mu}}

\newcommand{\normosepsmoom}[1]{\normos{#1}_{\om,\epsmo}}

\newcommand{\normosmumoom}[1]{\normos{#1}_{\om,\mumo}}

\newcommand{\normosmumooms}[1]{\normos{#1}_{\oms,\mumo}}

\DeclareMathOperator{\diam}{diam}

\usepackage{cite}
\usepackage{nicefrac}

\def\hio{\mathsf{X}}
\def\hit{\mathsf{Y}}
\def\A{\mathrm{A}}
\def\As{\A^{*}}
\def\cA{\mathcal{A}}
\def\cAs{\mathcal{A}^{*}}

\def\zet{\zeta}
\def\zets{\zet^{*}}
\def\J{J}
\def\e{j}
\def\eb{\bar{\e}}
\def\eh{\e_{\mathtt d}}

\def\et{\tilde{\e}}
\def\h{h}
\def\Hh{H_{\mathtt d}}
\def\Hti{\tilde{H}}
\renewcommand{\cE}{\mathscr{J}}
\newcommand{\cEa}{\cE}
\def\calA{\mathscr{A}}
\def\calB{\mathscr{B}}
\def\calC{\mathscr{C}}
\def\calD{\mathscr{D}}

\newcommand{\rn}{\mathbb{R}}
\renewcommand{\rt}{\rn^3}
\renewcommand{\rttt}{\rn^{3\times3}}
\newcommand{\ka}{\kappa}
\newcommand{\kamo}{\ka^{-1}}
\newcommand{\cM}{\calM}

\newcommand{\cMt}{\tilde{\calM}}

\newcommand{\ca}{c_{\A}}
\newcommand{\cas}{c_{\As}}
\def\cmtomepsmumo{c_{\mathtt{m,t},\om,\eps,\mumo}}

\def\chmom{\hat{c}_{\mathtt{m},\om}}
\def\cpom{c_{\mathtt{p},\om}}
\def\cpcom{c_{\mathtt{p},\circ,\om}}
\def\cpcomeps{c_{\mathtt{p},\circ,\om,\eps}}
\def\chpom{\hat{c}_{\mathtt{p},\om}}
\def\cpoms{c_{\mathtt{p},\oms}}
\def\cpomseps{c_{\mathtt{p},\oms,\eps}}
\def\chpoms{\hat{c}_{\mathtt{p},\oms}}

\newcommand{\cic}{\Cgen{\infty}{\circ}}
\newcommand{\cicom}{\cic(\om)}
\newcommand{\hooms}{\ho(\oms)}
\newcommand{\hob}{\Hgen{1}{\bot}{}}
\newcommand{\hoboms}{\Hgen{1}{\bot}{}(\oms)}

\renewcommand{\r}{\rotgen{}{}}
\renewcommand{\rc}{\rotgen{}{\circ}}
\renewcommand{\rz}{\rotgen{0}{}}
\newcommand{\rcz}{\rotgen{0}{\circ}}
\newcommand{\roms}{\r(\oms)}
\newcommand{\rcoms}{\rc(\oms)}
\newcommand{\rzoms}{\rz(\oms)}

\renewcommand{\d}{\divgen{}{}}
\newcommand{\dc}{\divgen{}{\circ}}
\newcommand{\dz}{\divgen{0}{}}
\newcommand{\dcz}{\divgen{0}{\circ}}

\newcommand{\dcoms}{\dc(\oms)}

\newcommand{\dczoms}{\dcz(\oms)}
\newcommand{\harmdeps}{\harm{}{\mathtt{D},\eps}{}}
\newcommand{\harmnmu}{\harm{}{\mathtt{N},\mu}{}}

\newcommand{\harmnepsoms}{\harm{}{\mathtt{N},\eps}{}(\oms)}

\renewcommand{\scpom}[2]{\scpsom{#1}{#2}}
\newcommand{\scpepsom}[2]{\scpsepsom{#1}{#2}}
\newcommand{\scpmuom}[2]{\scpsmuom{#1}{#2}}

\newcommand{\scpmumoom}[2]{\scpsmumoom{#1}{#2}}
\newcommand{\scpoms}[2]{\scpsoms{#1}{#2}}
\newcommand{\scpepsoms}[2]{\scpsepsoms{#1}{#2}}

\renewcommand{\norm}[1]{\normos{#1}}
\newcommand{\normepsom}[1]{\normosepsom{#1}}
\newcommand{\normmuom}[1]{\normosmuom{#1}}
\newcommand{\normepsmoom}[1]{\normosepsmoom{#1}}
\newcommand{\normmumoom}[1]{\normosmumoom{#1}}
\newcommand{\normoms}[1]{\norm{#1}_{\oms}}
\newcommand{\normepsoms}[1]{\normosepsoms{#1}}
\newcommand{\normmuoms}[1]{\normosmuoms{#1}}

\newcommand{\normmumooms}[1]{\normosmumooms{#1}}
\renewcommand{\dnorm}[1]{\norms{#1}}
\newcommand{\dnormrot}[1]{\dnorm{#1}_{\rot}}
\renewcommand{\tnorm}[1]{\tnorms{#1}}

\newcommand{\pic}{\overset{\circ}{\pi}}

\newcommand{\pioms}{\pi_{\oms}}

\begin{document}

% preprint series UDE Mathematik

%\thispagestyle{empty}
%\Large
%\begin{center}SCHRIFTENREIHE DER FAKULT\"AT F\"UR MATHEMATIK\end{center}
%\vspace*{5mm}
%\begin{center}
%A Posteriori Error Analysis\\ for the Optimal Control of Magneto-Static Fields
%\end{center}
%\vspace*{5mm}
%\begin{center}by\end{center}
%\begin{center}
%Dirk Pauly \&
%Irwin Yousept
%\end{center}
%\vspace*{5mm}
%\begin{center}SM-UDE-801\hspace{80mm}2016\end{center}
%\newpage
%\thispagestyle{empty}
%\vspace*{210mm}
%Received: ???, 2016
%\newpage
%\addtocounter{page}{-2}
%\normalsize

% preprint series UDE Mathematik

\date{\today}
\maketitle{}

\begin{abstract}
This paper is concerned with the analysis and numerical analysis for the optimal control 
of first-order magneto-static equations. 
Necessary and sufficient optimality conditions are established through a rigorous Hilbert space approach. 
Then, on the basis of the optimality system, 
we prove functional a posteriori error estimators for the optimal control, 
the optimal state, and the adjoint state.
3D numerical results illustrating the theoretical findings are presented.\\

\end{abstract}

%\tableofcontents

\section{Introduction}

Let $\emptyset\neq\oms\subset\om\subset\rt$ be bounded domains with boundaries $\gas:=\poms$, $\ga:=\pom$.
For simplicity, we assume that the boundaries $\gas$ and $\ga$ are Lipschitz
and satisfy $\dist(\gas,\ga)>0$, i.e., $\oms$ does not touch $\ga$.
Moreover, let material properties or constitutive laws $\eps,\mu:\om\to\rttt$ be given, 
which are symmetric, uniformly positive definite and belong to $\Liom$.
These assumptions are general throughout the paper. 
In our context, $\om$ denotes a large ``hold all'' computational domain. Therefore,
without loss of generality, we may assume that $\om$ is an open, bounded and convex set such as  a ball or a cube. 
On the other hand, the subdomain $\oms \subset \om$ represents a control region containing induction coils,  
where the applied current source control is acting. 
We underline that our analysis can be extended to the case, 
where $\oms$ is  non-connected with finite topology.

For a given desired   magnetic field $\Hh\in\Ltom$
and a given shift control $\eh\in\Ltoms$,
we look for the optimal applied current density  
in $\omega$ by solving the following minimization problem:
\begin{align}
\label{optcontprobintro}
\min_{\e\in\cEa}F(\e) 
:=\frac{1}{2}\int_{\om}\norm{\mu^{1/2}(H(\e)-\Hh)}^2 
+\frac{\ka}{2}\int_{\oms}\norm{\eps^{1/2}(\e-\eh)}^2,
\end{align}
where $H(\e)=H$ satisfies the first-order linear magneto-static boundary value problem:
\begin{align}
\label{constraint1}
\rot H&=\eps\pi(\zet\e+\J)&\text{in }\om,\\
\label{constraint2}
\div\mu H&=0&\text{in }\om,\\
\label{constraint3}
n\cdot\mu H&=0&\text{on }\ga,\\
\label{constraint4}
\mu H&\,\,\bot\harmnmuom.
\end{align}
In the setting of \eqref{optcontprobintro}, 
$\cEa$ denotes the admissible control set, which is assumed to be a nonempty and closed subspace of $\Ltoms$.  
Moreover, $\kappa>0$ is the control cost term, and $\J \in \Ltom$ represents a fixed external current density. 
In \eqref{constraint1}, we employ the extension by zero operator $\zet$ from $\oms$ to $\om$ 
as well as the $\Lt$-orthonormal projector $\pi$ onto the range of rotations. 
The precise definitions of these two operators will be given in next section.   
Furthermore, $\harmnmuom$ denotes the kernel of \eqref{constraint1}-\eqref{constraint3}, i.e.,
the set of all square integrable vector fields $H$ with
$\rot H=0$, $\div\mu H=0$ in $\om$ and $n\cdot\mu H=0$ on $\ga$,
where $n$ denotes the exterior unit normal to $\ga$.  
Let us also point out that \eqref{constraint1}-\eqref{constraint4} are understood in a weak sense.  

Using a rigorous Hilbert space approach for the state and adjoint state equations,
we derive necessary and sufficient optimality conditions for \eqref{optcontprobintro}.
Having established a  variational formulation for the corresponding optimality system,
we adjust this formulation for suitable numerical approximations and
prove functional a posteriori error estimates for the error in the optimal quantities based on
the spirit of Repin \cite{NeittaanmakiRepin2004,repinbookone}.
Finally, we propose a mixed formulation for computing the optimal control $\eb$
and   present some numerical results, which illustrate the efficiency of the proposed error estimator.

To the best of the authors' knowledge, this paper presents original contributions 
on the functional a posteriori error analysis for the optimal control 
of first-order magneto-static equations. 
We are only aware of the previous contributions \cite{hopyou14,xuzou16}  
on the residual a posteriori error analysis for optimal control problems 
based on the second-order magnetic vector potential formulation.  
For recent mathematical results in the optimal control of electromagnetic problems, 
we refer to \cite{kollang13a,kollang13b,nicstitro14,nicstitro15,troeval16,troyou12,you12a,you12b,you13}.

\section{Definitions and Preliminaries}

We do not distinguish in our notations between scalar functions or vector fields.
The standard $\Ltom$ inner product will be denoted by $\scpom{\,\cdot\,}{\,\cdot\,}$.
$\Ltepsom$ denotes $\Ltom$ equipped with the weighted inner product 
$\scpepsom{\,\cdot\,}{\,\cdot\,}:=\scpom{\eps\,\cdot\,}{\,\cdot\,}$
and for the respective norms we write $\normom{\,\cdot\,}$ and $\normepsom{\,\cdot\,}$.
All these definitions extend to $\mu$ as well as to $\oms$.
The standard Sobolev spaces and the corresponding Sobolev spaces for Maxwell's equations will be written as 
$\Hkom$ for $k\in\nz_{0}$ and
$$\rom:=\set{E\in\Ltom}{\rot E\in\Ltom},\quad
\dom:=\set{E\in\Ltom}{\div E\in\Ltom},$$
all equipped with the natural inner products and graph norms.
Moreover, for the sake of boundary conditions 
we define the Sobolev spaces $\Hkcom$ and $\rcom$, $\dcom$
as the closures of test functions or test vector fields from $\Cicom$ 
in the respective graph norms. A zero at the lower right corner of the Sobolev spaces
indicates a vanishing differential operator, e.g., 
$$\rzom=\set{E\in\rom}{\rot E=0},\quad\dczom=\set{E\in\dcom}{\div E=0}.$$
Furthermore, we introduce the spaces of Dirichlet and Neumann fields by
$$\harmdepsom:=\rczom\cap\epsmo\dzom,\quad\harmnmuom:=\rzom\cap\mumo\dczom.$$
All the defined spaces are Hilbert spaces
and all definitions extend to $\oms$ or generally to any domain as well.
We will omit the domain in our notations of the spaces if the underlying domain is $\om$.

It is well known that the embeddings
\begin{align}
\label{compemb}
\rc\cap\epsmo\d\hookrightarrow\Lt,\quad
\r\cap\epsmo\dc\hookrightarrow\Lt
\end{align}
are compact, see
\cite{weckmax,picardcomimb,picardweckwitschxmas,webercompmax,witschremmax,leisbook,jochmanncompembmaxmixbc,bauerpaulyschomburgmcpweaklip},
being a crucial point in the theory for Maxwell's equations.
By the compactness of the unit balls and a standard indirect argument we get immediately that
$\harmdeps$ and $\harmnmu$ are finite dimensional and that the well known 
Maxwell estimates, i.e., there exists $c>0$ such that
\begin{align}
\label{maxestrc}
\forall\,&E\in\rc\cap\epsmo\d\cap\harmdeps^{\bot_{\eps}}&
\normepsom{E}&\leq c\big(\normom{\rot E}^2+\normom{\div\eps E}^2\big)^{1/2},\\
\label{maxestr}
\forall\,&H\in\r\cap\mumo\dc\cap\harmnmu^{\bot_{\mu}}&
\normmuom{H}&\leq c\big(\normom{\rot H}^2+\normom{\div\mu H}^2\big)^{1/2},
\end{align}
hold, where $\bot$ resp. $\bot_{\eps}$ denotes orthogonality in $\Lt$ resp. $\Lteps$.
By the projection theorem and Hilbert space methods we have
$$\Lteps=\na\hoc\oplus_{\eps}\epsmo\dz=\rcz\oplus_{\eps}\epsmo\ol{\rot\r},\quad
\Ltmu=\na\ho\oplus_{\mu}\mumo\dcz=\rz\oplus_{\mu}\mumo\ol{\rot\rc}$$
with closures in $\Lt$. 
Here $\oplus$ resp. $\oplus_{\eps}$ denotes the orthogonal sum in $\Lt$ resp. $\Lteps$.
We note that by Rellich's selection theorem the ranges
$\na\hoc$ and $\na\ho$ are already closed. Therefore,
\begin{align}
\label{Helmholtzr}
\rc &=\rcz\oplus_{\eps}\big(\rc\cap\epsmo\ol{\rot\r}\big),&
\r&=\rz\oplus_{\mu}\big(\r\cap\mumo\ol{\rot\rc}\big)
\end{align}
and thus
\begin{align}
\label{Helmholtzrotr}
\rot\rc&=\rot\big(\rc\cap\epsmo\ol{\rot\r}\big),&
\rot\r&=\rot\big(\r\cap\mumo\ol{\rot\rc}\big)
\end{align}
hold. Since obviously $\ol{\rot\r}\subset\dz\cap\harmdeps^{\bot}$
and $\ol{\rot\rc}\subset\dcz\cap\harmnmu^{\bot}$,
we obtain by the Maxwell estimates \eqref{maxestrc} and \eqref{maxestr}
that all ranges of $\rot$ are also closed, i.e.,
$$\ol{\rot\rc}=\rot\rc=\rot\big(\rc\cap\epsmo\rot\r\big),\quad
\ol{\rot\r}=\rot\r=\rot\big(\r\cap\mumo\rot\rc\big).$$
Since $\na\hoc\subset\rcz$ and $\na\ho\subset\rz$ we have
$$\rcz=\na\hoc\oplus_{\eps}\harmdeps,\quad
\rz=\na\ho\oplus_{\mu}\harmnmu$$
and hence we get the general Helmholtz decompositions
\begin{align}
\label{Helmholtzgen}
\Lteps&=\na\hoc\oplus_{\eps}\harmdeps\oplus_{\eps}\epsmo\rot\r,&
\Ltmu&=\na\ho\oplus_{\mu}\harmnmu\oplus_{\mu}\mumo\rot\rc.
\end{align}
Note that we have analogously $\rot\rc\subset\dcz$ and $\rot\r\subset\dz$ and thus
$$\epsmo\dz=\epsmo\rot\r\oplus_{\eps}\harmdeps,\quad
\mumo\dcz=\mumo\rot\rc\oplus_{\mu}\harmnmu,$$
which gives again the Helmholtz decompositions \eqref{Helmholtzgen}.
At this point we introduce two orthonormal projectors 
\begin{align}
\label{projectordefinition}
\pi:\Lteps&\to\epsmo\rot\r\subset\Lteps,&
\pic:\Ltmu&\to\mumo\rot\rc\subset\Ltmu.
\end{align}
Note that the range of $\pi$ resp. $\pic$ equals $\epsmo\rot\r$ resp. $\mumo\rot\rc$
and that we have $\pi=\id$ resp. $\pic=\id$ on $\epsmo\rot\r$ resp. $\mumo\rot\rc$
and $\pi=0$ resp. $\pic=0$ on $\rcz$ resp. $\rz$.
Moreover, by \eqref{Helmholtzr} and \eqref{Helmholtzrotr} we see 
$\pi\rc=\rc\cap\epsmo\rot\r$ and $\pic\r=\r\cap\mumo\rot\rc$ and that
$\rot\pi E=\rot E$ and $\rot\pic H=\rot H$ hold 
for $E\in\rc$ and $H\in\r$. We also need the extension by zero operator
$$\Abb{\zet}{\Ltepsoms}{\Lteps}{\e}{\begin{cases}\e&\text{in }\oms\\0&\text{in }\om\setminus\ol{\oms}\end{cases}}.$$
Note that as orthonormal projectors $\pi:\Lteps\to\Lteps$ and $\pic:\Ltmu\to\Ltmu$ are selfadjoint
and that the adjoint of $\zet$ is the restriction operator $\zets=\,\cdot\,|_{\oms}:\Lteps\to\Ltepsoms$.
We also have $\zets\zet=\id$ on $\Ltepsoms$.
We emphasize that all our definitions and results from this section 
extend to $\oms$ or other domains as well.

For operators $\A$, here usually linear, 
we denote by $D(\A)$, $R(\A)$ and $N(\A)$ the domain of definition, the range 
and the kernel or null space of $\A$, respectively. 
For two Hilbert spaces $\hio$, $\hit$ and a densely defined and linear operator
$\A:D(\A)\subset\hio\to\hit$ we denote by $\As:D(\As)\subset\hit\to\hio$ 
its Hilbert space adjont.

\section{Functional Analytical Setting}

Let $\hio$, $\hit$ be two Hilbert spaces and let
\begin{align}
\label{Adef}
\A:D(\A)\subset\hio\to\hit
\end{align}
be a densely defined and closed linear operator with adjoint
\begin{align}
\label{Asdef}
\As:D(\As)\subset\hit\to\hio.
\end{align}
Equipping $D(\A)$ and $D(\As)$ with the respective graph norms
makes them Hilbert spaces.
By the projection theorem we have
\begin{align}
\label{projtheo1}
\hio&=N(\A)\oplus\ol{R(\As)},&D(\A)&=N(\A)\oplus\big(D(\A)\cap\ol{R(\As)}\big),\\
\label{projtheo2}
\hit&=N(\As)\oplus\ol{R(\A)},&D(\As)&=N(\As)\oplus\big(D(\As)\cap\ol{R(\A)}\big),
\end{align}
and
\begin{align}
\label{projtheo3}
N(\As)^{\bot_{\hit}}&=\ol{R(\A)},&R(\A)&=\A\big(D(\A)\cap\ol{R(\As)}\big),\\
\label{projtheo4}
N(\A)^{\bot_{\hio}}&=\ol{R(\As)},&R(\As)&=\As\big(D(\As)\cap\ol{R(\A)}\big).
\end{align}
Let us fix the crucial general assumption of this section:
The embedding
\begin{align}
\label{Agenass}
D(\A)\cap\ol{R(\As)}\hookrightarrow\hio
\end{align}
should be compact.

\begin{lem}
\label{Alem}
Assume \eqref{Agenass} holds. Then:
\begin{itemize}
\item[\bf(i)] 
$R(\A)$ and $R(\As)$ are closed.
\item[\bf(ii)] 
$\exists\,\ca>0\quad\forall\,x\in D(\A)\cap R(\As)\quad\norm{x}_{\hio}\leq\ca\norm{\A x}_{\hit}$
\item[\bf(ii')] 
$\exists\,\cas>0\quad\forall\,y\in D(\As)\cap R(\A)\quad\norm{y}_{\hit}\leq\cas\norm{\As y}_{\hio}$
\item[\bf(iii)] 
$D(\As)\cap R(\A)$ is compactly embedded into $\hit$.
\item[\bf(iii')] 
$D(\A)\cap R(\As)\hookrightarrow\hio\qequi D(\As)\cap R(\A)\hookrightarrow\hit$ 
\end{itemize}
\end{lem}

The lemma is  standard, but for convenience we give a simple and short proof.\\

\begin{proof}
First we show
\begin{align}
\label{Aestproof}
\exists\,\ca>0\quad\forall\,x\in D(\A)\cap\ol{R(\As)}\quad\norm{x}_{\hio}\leq\ca\norm{\A x}_{\hit}.
\end{align}
Let us assume that this is wrong. Then, there exists a sequence
$(x_{n})\subset D(\A)\cap\ol{R(\As)}$ with $\norm{x_{n}}_{\hio}=1$ and $\norm{\A x}_{\hit}\to0$.
Hence, $(x_{n})$ is bounded in $D(\A)\cap\ol{R(\As)}$ and we can extract a subsequence, again denoted by $(x_{n})$,
with $x_{n}\xrightarrow{\hio}x\in\hio$. Since $\A$ is closed, 
$x$ belongs to $N(\A)\cap\ol{R(\As)}=\{0\}$, a contradiction,
because $1=\norm{x_{n}}_{\hio}\to\norm{x}_{\hio}=0$.

Now, let $y\in\ol{R(\A)}$, i.e., $y\in\ol{\A\big(D(\A)\cap\ol{R(\As)}\big)}$ by \eqref{projtheo3}.
Hence, there exists a sequence $(x_{n})$ in $D(\A)\cap\ol{R(\As)}$
with $\A x_{n}\xrightarrow{\hit}y$. By \eqref{Aestproof}, $(x_{n})$ is a Cauchy sequence in $D(\A)$
and thus $x_{n}\xrightarrow{D(\A)}x\in D(\A)$. Especially $\A x_{n}\to\A x$ implies $y=\A x\in R(\A)$.
Therefore, $R(\A)$ is closed. By the closed range theorem, 
see e.g. \cite[VII, 5]{yosidabook}, $R(\As)$ is closed as well. 
This proves (i) and together with \eqref{Aestproof} also (ii) is proved. 

Let $(y_{n})$ be a bounded sequence in $D(\As)\cap R(\A)$.
By \eqref{projtheo3}, $y_{n}\in\A\big(D(\A)\cap R(\As)\big)$
and there exists a sequence $(x_{n})\subset D(\A)\cap R(\As)$
with $\A x_{n}=y_{n}$. By (ii), $(x_{n})$ is bounded in $D(\A)\cap R(\As)$.
Hence, without loss of generality, $(x_{n})$ converges in $\hio$.
Then, for $x_{n,m}:=x_{n}-x_{m}$ and $y_{n,m}:=y_{n}-y_{m}$ we have
$$\norm{y_{n,m}}_{\hit}^2
=\scp{\A x_{n,m}}{y_{n,m}}_{\hit}
=\scp{x_{n,m}}{\As y_{n,m}}_{\hio}
\leq c\norm{x_{n,m}}_{\hio}.$$
Therefore, $(y_{n})$ is a Cauchy sequence in $\hit$, showing (iii).

Now, (ii') follows by (iii) analogously to the proof of (ii).
(iii') is clear by duality since $(\A,\As)$ is a `dual pair', i.e., $\A^{**}=\bar{\A}=\A$,
where $\bar{\A}$ denotes the closure of $\A$.
\end{proof}

\begin{rem}
\label{Aremconstants}
The best constants in Lemma \ref{Alem} (ii) and (ii') are even equal, i.e.,
$$\frac{1}{\ca}=\inf_{0\neq x\in D(\A)\cap R(\As)}\frac{\norm{\A x}_{\hit}}{\norm{x}_{\hio}}
=\inf_{0\neq y\in D(\As)\cap R(\A)}\frac{\norm{\As y}_{\hio}}{\norm{y}_{\hit}}=\frac{1}{\cas}.$$
See \cite[Theorem 2]{paulymaxconst2} and also \cite{paulymaxconst0,paulymaxconst1}.
\end{rem}

Since the decompositions \eqref{projtheo1} and \eqref{projtheo2} reduce $\A$ and $\As$,
we obtain that the adjoint of the reduced operator
\begin{align}
\label{cAdef}
\Abb{\cA}{D(\cA):=D(\A)\cap R(\As)\subset R(\As)}{R(\A)}{x}{\A x}
\end{align}
is given by the reduced adjoint operator
\begin{align}
\label{cAsdef}
\Abb{\cAs}{D(\cAs):=D(\As)\cap R(\A)\subset R(\A)}{R(\As)}{y}{\As y}.
\end{align}

We immediately get by Lemma \ref{Alem} the following.

\begin{lem}
\label{Alemmore}
It holds:
\begin{itemize}
\item[\bf(i)] 
$R(\cA)=R(\A)$ and $R(\cAs)=R(\As)$.
\item[\bf(ii)] 
$\cA$ and $\cAs$ are injective and
$\cA^{-1}:R(\A)\to D(\cA)$ and $(\cAs)^{-1}:R(\As)\to D(\cAs)$ continuous.
\item[\bf(ii')] 
As operators on $R(\A)$ and $R(\As)$,
$\cA^{-1}:R(\A)\to R(\As)$ and $(\cAs)^{-1}:R(\As)\to R(\A)$ are compact.
\end{itemize}
\end{lem}

Let us now transfer these results to Maxwell's equations.
We set $\hio:=\Lteps$ and $\hit:=\Ltmu$. It is well known that
\begin{align*}
&\Abb{\A}{D(\A)\subset\Lteps}{\Ltmu}{E}{\mumo\rot E},&
D(\A)&:=\rc,&
R(\A)&=\mumo\rot\rc,
\intertext{is a densely defined and closed linear operator with adjoint}
&\Abb{\As}{D(\As)\subset\Ltmu}{\Lteps}{H}{\epsmo\rot H},&
D(\As)&\,=\r,&
R(\As)&=\epsmo\rot\r.
\end{align*}
By e.g. the first compact embedding of \eqref{compemb}, i.e,
$\rc\cap\epsmo\d\hookrightarrow\Lt$, we get \eqref{Agenass}, i.e., 
$$\rc\cap\ol{\epsmo\rot\r}\subset\rc\cap\epsmo\dz\subset\rc\cap\epsmo\d\hookrightarrow\Lteps.$$
Hence, $\rot\rc$ and $\rot\r$ are closed and we have the Maxwell estimates
\begin{align}
\label{maxrotestone}
\forall\,&E\in\rc\cap\epsmo\rot\r&
\normepsom{E}&\leq\ca\normmuom{\mumo\rot E},\\
\label{maxrotesttwo}
\forall\,&H\in\r\cap\mumo\rot\rc&
\normmuom{H}&\leq\cas\normepsom{\epsmo\rot H}.
\end{align}
\eqref{projtheo1}-\eqref{projtheo4} provide partially 
the Helmholtz decompositions from the latter section, i.e,
\begin{align*}
\Lteps&=\rcz\oplus_{\eps}\epsmo\rot\r,&\rc&=\rcz\oplus_{\eps}\big(\rc\cap\epsmo\rot\r\big),\\
\Ltmu&=\rz\oplus_{\mu}\mumo\rot\rc,&\r&=\rz\oplus_{\mu}\big(\r\cap\mumo\rot\rc\big),\\
\rz^{\bot_{\mu}}&=\mumo\rot\rc,&\mumo\rot\rc&=\mumo\rot\big(\rc\cap\epsmo\rot\r\big),\\
\rcz^{\bot_{\eps}}&=\epsmo\rot\r,&\epsmo\rot\r&=\epsmo\rot\big(\r\cap\mumo\rot\rc\big).
\end{align*}
The injective operators $\cA$ and $\cAs$ are
\begin{align*}
&\Abb{\cA}{D(\cA)\subset\epsmo\rot\r}{\mumo\rot\rc}{E}{\mumo\rot E},&
D(\cA)&:=\rc\cap\epsmo\rot\r,\\
&\Abb{\cAs}{D(\cAs)\subset\mumo\rot\rc}{\epsmo\rot\r}{H}{\epsmo\rot H},&
D(\cAs)&\,=\r\cap\mumo\rot\rc
\end{align*}
with 
$$R(\cA)=R(\A)=\mumo\rot\rc=R(\pic),\qquad
R(\cAs)=R(\As)=\epsmo\rot\r=R(\pi).$$
The inverses 
\begin{align*}
\cA^{-1}:\mumo\rot\rc&\to\rc\cap\epsmo\rot\r,&
(\cAs)^{-1}:\epsmo\rot\r&\to\r\cap\mumo\rot\rc,\\
\cA^{-1}:\mumo\rot\rc&\to\epsmo\rot\r,&
(\cAs)^{-1}:\epsmo\rot\r&\to\mumo\rot\rc
\end{align*}
are continuous and compact, respectively.
We note again that both $D(\cA)$ and $D(\cAs)$ are compactly embedded into $\Lt$.

\section{The Optimal Control Problem}

We start by formulating our optimal control problem 
\eqref{optcontprobintro}-\eqref{constraint4} in a proper Hilbert space setting.
As mentioned in the introduction, the admissible control set 
 $\cEa$ is assumed to be 
a nonempty and closed subspace of $\Ltepsoms$.
For some given $\J\in\Lteps$, $\Hh\in\Ltmu$ and $\eh\in\Ltepsoms$ let us define
\begin{align}
\label{cEprojdef}
\pioms:\Ltepsoms\to\cE,
\end{align}
the $\Ltepsoms$ orthonormal projector onto $\cE$.
Moreover, we introduce the norm $\tnorm{\,\cdot\,}$ by
$$\tnorm{(\Phi,\phi)}^2
:=\normmuom{\Phi}^2
+\ka\normepsoms{\phi}^2,\quad
(\Phi,\phi)\in\Ltmu\times\Ltepsoms,$$
and the quadratic functional $F$ by
\begin{align}
\label{functdef}
\Abb{F}{\Ltepsoms}{[0,\infty)}{\e}{\ds\frac{1}{2}\tnorm{(H(\e)-\Hh,\e-\eh)}^2},
\end{align}
i.e.,
$$F(\e)
=\frac{1}{2}\tnorm{(H(\e)-\Hh,\e-\eh)}^2
=\frac{1}{2}\normmuom{H(\e)-\Hh}^2+\frac{\ka}{2}\normepsoms{\e-\eh}^2,$$
where $H=H(\e)$ is the unique solution of the magneto static problem \eqref{constraint1}-\eqref{constraint4},
which can be formulated as
\begin{align}
\label{Hdef}
H&\in\r\cap\big(\mumo\rot\rc\big),&
\epsmo\rot H&=\pi(\zet\e+\J).
\end{align}
We note that by $\pi(\zet\e+\J)\in\epsmo\rot\r$ and by \eqref{Helmholtzrotr}, i.e.,
$\rot\r=\rot\big(\r\cap\mumo\rot\rc\big)$, \eqref{Hdef} is solvable
and the solution is unique, since 
$$\rz\cap\big(\mumo\rot\rc\big)=\rz\cap\mumo\dcz\cap\harmnmu^{\bot_{\mu}}
=\harmnmu\cap\harmnmu^{\bot_{\mu}}=\{0\}.$$
Moreover, the solution operator,
mapping the pair $(\e,\J)\in\Ltepsoms\times\Lteps$ to 
$H\in\r\cap\big(\mumo\rot\rc\big)$,
is continuous since by \eqref{maxestr} or \eqref{maxrotesttwo}
(with generic constants $c>0$)
$$\norm{H}_{\r}
=\big(\normom{H}^2+\normom{\rot H}^2\big)^{1/2}
\leq c\normepsom{\pi(\zet\e+\J)}
\leq c\normepsom{\zet\e+\J}
\leq c\big(\normepsoms{\e}+\normepsom{\J}\big).$$
We note that the unique solution is given by $H:=H(\e):=(\cAs)^{-1}\pi(\zet\e+\J)$
depending affine linearly and continuously on $\e\in\Ltepsoms$.

Now, our optimal control problem 
\eqref{optcontprobintro}-\eqref{constraint4} reads as follows:
Find $\eb\in\cEa$, such that
\begin{align}
\label{optcontprob1}
F(\eb)=\min_{\e\in\cEa}F(\e),
\end{align}
subject to $H(\e)\in\r\cap\big(\mumo\rot\rc\big)$ and $\epsmo\rot H(\e)=\pi(\zet\e+\J)$.
Another equivalent formulation using the Hilbert space operators from the latter section
and $R(\pi)=\epsmo\rot\r=R(\cAs)$ is:
Find $\eb\in\cEa$, such that
\begin{align}
\label{optcontprob2}
F(\eb)=\min_{\e\in\cEa}F(\e),
\end{align}
subject to $H(\e)\in D(\cAs)$ and $\cAs H(\e)=\pi(\zet\e+\J)$.
Our last formulation is:
Find $\eb\in\cEa$, such that
\begin{align}
\label{optcontprob3}
F(\eb)=\min_{\e\in\cEa}F(\e),\quad
F(\e)=\frac{1}{2}\normmuom{(\cAs)^{-1}\pi(\zet\e+\J)-\Hh}^2
+\frac{\ka}{2}\normepsoms{\e-\eh}^2.
\end{align}

Let us now focus on the latter formulation \eqref{optcontprob3}.
Since $(\cAs)^{-1}\pi(\zet\e+\J)\in R(\A)=R(\pic)$ 
and $\e\in R(\pioms)=\cEa$ we have
$$F(\e)
=\frac{1}{2}\normmuom{(\cAs)^{-1}\pi(\zet\e+\J)-\pic\Hh}^2
+\frac{\ka}{2}\normepsoms{\e-\pioms\eh}^2
+\frac{1}{2}\normmuom{(1-\pic)\Hh}^2
+\frac{\ka}{2}\normepsoms{(1-\pioms)\eh}^2$$
and hence we may assume from now on without loss of generality
\begin{align}
\label{genass}
\begin{split}
\Hh&=\pic\Hh\in R(\A)=R(\pic)=\mumo\rot\rc,\quad
\J=\pi\J\in R(\As)=R(\pi)=\epsmo\rot\r,\\
\eh&=\pioms\eh\in R(\pioms)=\cEa.
\end{split}
\end{align}

\begin{lem}
\label{optcontexistandunique}
The optimal control problem \eqref{optcontprob3} admits a unique solution
$\eb\in\cEa$. Moreover, $\eb\in\cEa$ is the unique solution of 
  \eqref{optcontprob3},
if and only if $\eb\in\cEa$ is the unique solution of $F'(\eb)=0$.
\end{lem}

\begin{proof}
$(\cAs)^{-1}\pi\zet$ is linear and continuous and $F$ 
is convex and differentiable.
Since $\emptyset\neq\cEa$ is a closed subspace, the assertions follow immediately.
\end{proof}

Let us compute the derivative.
Since $(\cAs)^{-1}\pi\zet$ is linear and continuous we have for all $\e,\h\in\Ltepsoms$
\begin{align*} 
F'(\e)\h
&=\scpmuom{(\cAs)^{-1}\pi(\zet\e+\J)-\Hh}{(\cAs)^{-1}\pi\zet\h}+\ka\scpepsoms{\e-\eh}{\h}\\
&=\scpepsoms{\zets\pi\cA^{-1}((\cAs)^{-1}\pi(\zet\e+\J)-\Hh)+\ka(\e-\eh)}{\h}\\
&=\scpepsoms{\zets\cA^{-1}((\cAs)^{-1}\pi(\zet\e+\J)-\Hh)+\ka(\e-\eh)}{\h}.
\intertext{Hence, for all $\e,\h\in\cE$, we have}
F'(\e)\h
&=\scpepsoms{\zets\cA^{-1}((\cAs)^{-1}\pi(\zet\e+\J)-\Hh)+\ka(\e-\eh)}{\pioms\h}\\
&=\scpepsoms{\pioms\zets\cA^{-1}((\cAs)^{-1}\pi(\zet\e+\J)-\Hh)+\ka\pioms(\e-\eh)}{\h}\\
&=\scpepsoms{\pioms\zets\cA^{-1}((\cAs)^{-1}\pi(\zet\e+\J)-\Hh)+\ka(\e-\eh)}{\h}.
\end{align*}

In view of this formula and Lemma \ref{optcontexistandunique}, 
we obtain the following necessary and sufficient optimality system:

\begin{theo}
\label{optconttheo}
$\eb\in\cEa$ is the unique optimal control of \eqref{optcontprob3},
if and only if $(\eb,\Hb,\Eb)\in\cEa\times D(\cAs)\times D(\cA)$ 
is the unique solution of 
\begin{equation}
\label{varineq1}
\eb=\eh-\frac{1}{\ka}\pioms\zets\Eb,\quad
\Eb=\cA^{-1}(\Hb-\Hh),\quad
\Hb=(\cAs)^{-1}\pi(\zet\eb+\J).
\end{equation}
%Note $\Hb=\cS\eb$ and $\Eb=\cSs\eb=\cA^{-1}(\cS\eb-\Hh)$.
\end{theo}

%\begin{proof}
%We just need to show $\ec=\eb$, for any solution $\ec$ of 
%$$\ec=\eh-\frac{1}{\ka}\pioms\zets\Ec,\quad
%\Ec=\cA^{-1}(\Hc-\Hh),\quad
%\Hc=(\cAs)^{-1}\pi(\zet\ec+\J).$$
%Of course, $\ec\in\cEa$ and $\Ec\in D(\cA)$ and $\Hc\in D(\cAs)$ are well defined.
%Moreover, $\Ec=\cSs(\ec)$.
%Hence, $\ec$ solves \eqref{varineq2} and therefore $\ec=\eb$.
%\end{proof}

%\begin{rem}
%\label{optconttheorem}
%We can formulate Theorem \ref{optconttheo} in a more PDE-like fashion:
%The unique optimal control $\eb\in\cEa$ of the optimal control problem 
%\eqref{optcontprob1} or \eqref{optcontprob2} or \eqref{optcontprob3}
%is given by \eqref{varineq1} or \eqref{varineq2}, the unique solution $\eb\in\cEa$ 
%of the optimality system 
%$$\eb=\eh-\frac{1}{\ka}\pioms\zets\Eb,\quad
%\rot\Eb=\mu(\Hb-\Hh),\quad
%\rot\Hb=\eps\pi\zet\eb+\eps\J$$
%with unique $\Eb\in\rc\cap\epsmo\rot\r$ and $\Hb\in\r\cap\mumo\rot\rc$.
%Note that $\Eb\in\rc$ and $\Hb\in\r$ are the unique solutions of the two systems 
%\begin{align*}
%\rot\Hb&=\eps\pi\zet\eb+\eps\J,&\rot\Eb&=\mu(\Hb-\Hh)&\text{in }\om,\\
%\div\mu\Hb&=0,&\div\eps\Eb&=0&\text{in }\om,\\
%n\cdot\mu\Hb&=0,&n\times\Eb&=0&\text{on }\ga,\\
%\mu\Hb&\,\,\bot\harmnmu,&\eps\Eb&\,\,\bot\harmdeps.
%\end{align*}
%\end{rem}

\begin{rem}
\label{optconttheorem}
The latter optimality system \eqref{varineq1} is equivalent to the following system: Find 
$(\eb,\Hb,\Eb)$ in $\cEa\times(\r\cap\mumo\rot\rc)\times(\rc\cap\epsmo\rot\r)$ such that
\begin{align*}
\rot\Hb&=\eps\pi\zet\eb+\eps\J,&\rot\Eb&=\mu(\Hb-\Hh)&\text{in }\om,\\
\div\mu\Hb&=0,&\div\eps\Eb&=0&\text{in }\om,\\
n\cdot\mu\Hb&=0,&n\times\Eb&=0&\text{on }\ga,\\
\mu\Hb&\,\,\bot\harmnmu,&\eps\Eb&\,\,\bot\harmdeps
\end{align*}
and $\eb=\eh-\frac{1}{\ka}\pioms\zets\Eb$.
\end{rem}

Now, we have different options to specify the projector $\pioms:\Ltepsoms\to\cE$.
The only restriction is that $\cE=\pioms\Ltepsoms$
is a nonempty and closed subspace of $\Ltepsoms$.
Let us recall suitable Helmholtz decompositions for $\Ltepsoms$
\begin{align}
\label{Helmdecooms}
\begin{split}
\Ltepsoms
=\rzoms\oplus_{\eps}\epsmo\rot\rcoms
&=\na\hooms\oplus_{\eps}\epsmo\dczoms\\
&=\na\hooms\oplus_{\eps}\harmnepsoms\oplus_{\eps}\epsmo\rot\rcoms.
\end{split}
\end{align}
For example, we can choose
\begin{itemize}
\item[\bf(i)] $\pioms=\id_{\Ltepsoms}$,
\item[\bf(ii)] $\pioms:\Ltepsoms\to\epsmo\rot\rcoms\subset\Ltepsoms$, 
the $\Ltepsoms$-orthonormal projector onto $\epsmo\rot\rcoms$ in the Helmholtz decompositions \eqref{Helmdecooms},
\item[\bf(iii)] $\pioms:\Ltepsoms\to\epsmo\dczoms\subset\Ltepsoms$, 
the $\Ltepsoms$-orthonormal projector onto $\epsmo\dczoms$ in the Helmholtz decompositions \eqref{Helmdecooms}.
\end{itemize}

For physical and numerical reasons it makes sense to choose (iii), i.e.,
\begin{align}
\label{piomschoice}
\pioms:\Ltepsoms\to\epsmo\dczoms=:\cE,
\end{align}
which will be assumed from now on.
We note that all our subsequent results hold for the choice (ii) as well.
Now, we derive an equation for the adjoint state $\Eb$.
By Theorem \ref{optconttheo}, $\Eb$
and our optimal control $\eb=\eh-\kamo\pioms\zets\Eb$ satisfy
for all $\Phi\in D(\A)$
\begin{align}
\label{vareqEneb1}
\begin{split}
\scpmuom{\A\Eb}{\A\Phi}
&=\scpmuom{\Hb-\Hh}{\A\Phi}
=\scpepsom{\As\Hb}{\Phi}
-\scpmuom{\Hh}{\A\Phi}\\
&=\scpepsom{\pi\zet\eb}{\Phi}
+\scpepsom{\J}{\Phi}
-\scpmuom{\Hh}{\A\Phi}.
\end{split}
\end{align}
Note that, in case of $\Phi\in D(\cA)\subset R(\As)=R(\pi)$ we can skip the projector $\pi$, i.e.,
$$\scpepsom{\pi\zet\eb}{\Phi}
=\scpepsom{\zet\eb}{\pi\Phi}
=\scpepsom{\zet\eb}{\Phi}
=\scpepsoms{\eb}{\zets\Phi}
=\scpepsoms{\eh}{\zets\Phi}
-\frac{1}{\ka}\scpepsoms{\pioms\zets\Eb}{\zets\Phi}.$$
Hence, for all $\Phi\in D(\cA)$
\begin{align}
\label{varformcAeq}
\scpmuom{\A\Eb}{\A\Phi}
+\frac{1}{\ka}\scpepsoms{\pioms\zets\Eb}{\pioms\zets\Phi}
=\scpepsoms{\eh}{\zets\Phi}
+\scpepsom{\J}{\Phi}
-\scpmuom{\Hh}{\A\Phi}.
\end{align}

\begin{rem}
\label{varformcA}
The latter variational formulation \eqref{varformcAeq}
admits a unique solution $E$ in $D(\cA)$ depending continuously on $\J$, $\Hh$ and $\eh$, i.e.,
$\norm{E}_{D(\A)}\leq c(\normom{\Hh}+\normoms{\eh}+\normom{\J})$.
This is clear by the Lax-Milgram lemma, since
the left hand side is coercive over $D(\cA)$, i.e., 
by Lemma \ref{Alem} (ii) for all $E\in D(\cA)$
$$\normmuom{\A E}^2
+\kamo\normepsoms{\pioms\zets E}^2
\geq\normmuom{\A E}^2
\geq c\norm{E}_{D(\A)}^2.$$
\end{rem}

For numerical reasons, it is not practical to work in $D(\cA)=D(\A)\cap R(\As)$. 
On the other hand, it is important to get rid of $\pi$ since
the numerical implementation of $\pi$ is a difficult task.
Fortunately, due to the choice of $\cE$ we have:

\begin{lem}
\label{withoutpi}
$\pi\zet\pioms=\zet\pioms$
\end{lem}

Note that this lemma would fail with the option (i) for $\pioms$.\\

\begin{proof}
Let $\e\in R(\pioms)=\epsmo\dczoms$.
Then, for any ball $B$ with $\om\subset B$
we have $\zet\eps\e\in\dcz$
and hence $\zet_{B}\zet\eps\e\in\dcz(B)$,
where $\zet_{B}$ denotes the extension by zero from $\om$ to $B$.
As $B$ is simply connected, there are no Neumann fields in $B$
yielding $\dcz(B)=\rot\rc(B)$. Thus, there exists $E\in\rc(B)$
with $\rot E=\zet_{B}\zet\eps\e$. But then the restriction 
$\zet_{B}^{*}E$ belongs to $\r$ and we have
$\rot\zet_{B}^{*}E=\zet_{B}^{*}\rot E=\zet\eps\e$
showing $\zet\e\in\epsmo\rot\r=R(\pi)$.
Hence, $\pi\zet\e=\zet\e$,
finishing the proof.
\end{proof}

Utilizing Lemma \ref{withoutpi} and $\eb\in R(\pioms)$
we obtain $\pi\zet\eb=\zet\eb$. 
Therefore, \eqref{vareqEneb1} turns into
\begin{align}
\nonumber
\forall\,\Phi&\in D(\A)&
\scpmuom{\A\Eb}{\A\Phi}
-\scpepsom{\zet\eb}{\Phi}
&=\scpepsom{\J}{\Phi}
-\scpmuom{\Hh}{\A\Phi}
\intertext{or equivalently with $\scpepsom{\zet\eb}{\Phi}=\scpepsoms{\eb}{\zets\Phi}$}
\nonumber
\forall\,\Phi&\in D(\A)&
\scpmuom{\A\Eb}{\A\Phi}
+\frac{1}{\ka}\scpepsoms{\pioms\zets\Eb}{\zets\Phi}
&=\scpepsoms{\eh}{\zets\Phi}
+\scpepsom{\J}{\Phi}
-\scpmuom{\Hh}{\A\Phi}.
\intertext{Hence, we obtain the following symmetric variational formulation for $\Eb\in D(\cA)$}
\label{vareqEneb2}
\forall\,\Phi&\in D(\A)&
\scpmuom{\A\Eb}{\A\Phi}
+\frac{1}{\ka}\scpepsoms{\pioms\zets\Eb}{\pioms\zets\Phi}
&=\scpepsom{\zet\eh+\J}{\Phi}
-\scpmuom{\Hh}{\A\Phi}.
\end{align}
By $\scpepsoms{\pioms\zets\Eb}{\pioms\zets\Phi}=\scpepsom{\zet\pioms\zets\Eb}{\Phi}$
and \eqref{vareqEneb2} we get immediately
$$\A\Eb+\Hh\in D(\As),\quad
\As(\A\Eb+\Hh)=\zet(\eh-\frac{1}{\ka}\pioms\zets\Eb)+\J.$$
Therefore, if $\Hh\in D(\As)$, then $\A\Eb\in D(\As)$
and we obtain in $\om$ the strong equation
\begin{align}
\label{strongeqEb}
\As\A\Eb+\frac{1}{\ka}\zet\pioms\zets\Eb
&=\zet\eh+\J-\As\Hh.
\end{align}
Translated to the PDE language \eqref{vareqEneb2} and \eqref{strongeqEb} 
read as follows: $\Eb\in\rc\cap\epsmo\rot\r$ with
\begin{align}
\label{vareqEneb2pde}
\forall\,\Phi&\in\rc&
\scpmumoom{\rot\Eb}{\rot\Phi}
+\frac{1}{\ka}\scpepsoms{\pioms\zets\Eb}{\pioms\zets\Phi}
&=\scpepsom{\zet\eh+\J}{\Phi}
-\scpom{\Hh}{\rot\Phi}
\end{align}
or, if $\Hh\in\r$,
\begin{align}
\label{strongeqEbpde}
\rot\mumo\rot\Eb+\frac{1}{\ka}\eps\zet\pioms\zets\Eb
&=\eps\zet\eh+\eps\J-\rot\Hh.
\end{align}

\begin{theo}
\label{equivareq}
For $\eb\in\Ltepsoms$ the following statements are equivalent:
\begin{itemize}
\item[\bf(i)] $\eb\in\cEa$ is the unique optimal control 
of the optimal control problem \eqref{optcontprob3}.
\item[\bf(ii)] $\eb$ is the unique solution of the optimality system 
$$\eb=\eh-\frac{1}{\ka}\pioms\zets\Eb,\quad
\Eb=\cA^{-1}(\Hb-\Hh),\quad
\Hb=(\cAs)^{-1}(\zet\eb+\J).$$
We note $\zet\eb=\pi\zet\eb$ by Lemma \ref{withoutpi} and $\eb\in\cEa$.
\item[\bf(iii)] $\eb=\eh-\kamo\pioms\zets\Eb$ 
and $\Eb\in D(\cA)$ satisfies \eqref{vareqEneb2}, i.e.,
$$\forall\,\Phi\in D(\A)\qquad
\scpmuom{\A\Eb}{\A\Phi}
+\frac{1}{\ka}\scpepsoms{\pioms\zets\Eb}{\pioms\zets\Phi}
=\scpepsom{\zet\eh+\J}{\Phi}
-\scpmuom{\Hh}{\A\Phi}.$$
\end{itemize}
By (iii), \eqref{vareqEneb2} is uniquely solvable.
\end{theo}

\begin{proof}
By Theorem \ref{optconttheo} we have (i)$\equi$(ii). 
Moreover, (ii)$\impl$(iii) follows from the previous considerations.
Hence, it remains to show (iii)$\impl$(ii).
For this, let $\e:=\eh-\kamo\pioms\zets E\in\cEa$ 
with $E\in D(\cA)$ satisfying
$$\forall\,\Phi\in D(\A)\quad
\scpmuom{\A E}{\A\Phi}
+\frac{1}{\ka}\scpepsoms{\pioms\zets E}{\pioms\zets\Phi}
=\scpepsom{\zet\eh+\J}{\Phi}
-\scpmuom{\Hh}{\A\Phi}.$$
Hence
$$H:=\A E+\Hh\in D(\As)\cap R(\A)=D(\cAs),\quad
\As H=\zet(\eh-\kamo\pioms\zets E)+\J.$$
Thus, $E\in D(\cA)$ solves $\A E=H-\Hh$ and
$H\in D(\cAs)$ solves $\As H=\zet\e+\J$.
Therefore, $E=\cA^{-1}(H-\Hh)$ and $H=(\cAs)^{-1}(\zet\e+\J)$,
so the tripple $(\e,E,H)$ solves the optimality system (ii), yielding $\e=\eb$.
\end{proof}

\section{Suitable Variational Formulations}
\label{secvarfrom}

Let us summarize the results optioned so far
and introduce some new notation.
We recall our choice \eqref{piomschoice}, i.e.,
$$\pioms:\Ltepsoms\to\epsmo\dczoms=\cE,$$
and the related Helmholtz decomposition
\begin{align}
\label{helmoms}
\Ltepsoms
=\na\hooms\oplus_{\eps}\cE.
\end{align}
Our aim is still to find and compute the optimal control $\eb\in\cEa$, such that
\begin{align}
\label{optcontprob4}
F(\eb)=\min_{\e\in\cEa}F(\e),\quad
F(\e)
=\frac{1}{2}\tnorm{(H(\e)-\Hh,\e-\eh)}^2
=\frac{1}{2}\normmuom{H(\e)-\Hh}^2
+\frac{\ka}{2}\normepsoms{\e-\eh}^2
\end{align}
subject to 
$$H(\e)\in\r\cap\big(\mumo\rot\rc\big),\quad
\epsmo\rot H(\e)=\pi\zet\e+\J=\zet\e+\J$$
by Lemma \ref{withoutpi}, where the right hand side,
the `desired' magnetic field and current density satisfy
$$\J\in R(\pi)=\epsmo\rot\r,\quad
\Hh\in R(\pic)=\mumo\rot\rc,\quad
\eh\in R(\pioms)=\cEa,$$
respectively. Moreover, $H=H(\e)$ solves the system
\begin{align*}
\rot H&=\eps(\zet\e+\J)&\text{in }\om,\\
\div\mu H&=0&\text{in }\om,\\
n\cdot\mu H&=0&\text{on }\ga,\\
\mu H&\,\,\bot\harmnmu,
\end{align*}
in a standard weak sense.

From now on, we assume generally that $\om$ is bounded and \emph{convex}.
Later, $\om$ will be a cube.
Since $\om$ is convex, it has a connected boundary and hence there are no Dirichlet fields, i.e., $\harmdeps=\{0\}$,
which is important for our variational formulations, as we will see later.
Note that also the Neumann fields vanish, i.e., $\harmnmu=\{0\}$,
because a convex domain is simply connected. 
We also recall Theorem \ref{optconttheo}, Remark \ref{optconttheorem} and  \eqref{piomschoice},
which we summarize in the following strong PDE-formulation:

\begin{theo}
\label{equivareqpde}
For $\eb\in\Ltepsoms$ the following statements are equivalent:
\begin{itemize}
\item[\bf(i)] $\eb\in\cE$ is the unique optimal control 
of the optimal control problem \eqref{optcontprob2}.
\item[\bf(ii)] $\eb$ is the unique solution of the optimality system 
$$\eb=\eh-\kamo\pioms\zets\Eb,\quad
\rot\Eb=\mu(\Hb-\Hh),\quad
\rot\Hb=\eps(\zet\eb+\J)$$
with unique $\Eb\in\rc\cap\epsmo\rot\r$ and $\Hb\in\r\cap\mumo\rot\rc$.
\item[\bf(iii)] $\eb=\eh-\kamo\pioms\zets\Eb$ 
and $\Eb$ is the unique solution of $\Eb\in\rc\cap\epsmo\rot\r$ satisfying
$$\forall\,\Phi\in\rc\qquad
\scpmumoom{\rot\Eb}{\rot\Phi}
+\kamo\scpepsoms{\pioms\zets\Eb}{\pioms\zets\Phi}
=\scpepsom{\zet\eh+\J}{\Phi}
-\scpom{\Hh}{\rot\Phi}.$$
\end{itemize}
\end{theo}

We note that by Remark \ref{varformcA} the variational formulation
\begin{align*}
\forall\,\Phi&\in\rc\cap\epsmo\rot\r&
\scpmumoom{\rot E}{\rot\Phi}
+\kamo\scpepsoms{\pioms\zets E}{\pioms\zets\Phi}
&=\scpepsom{\zet\eh+\J}{\Phi}
-\scpom{\Hh}{\rot\Phi}
\end{align*}
admits a unique solution $E\in\rc\cap\epsmo\rot\r$ 
depending continuously on the right hand side data, i.e.,
$\norm{E}_{\r}\leq c(\normom{\Hh}+\normoms{\eh}+\normom{\J})$.
The crucial point for applying the Lax-Milgram lemma 
is the Maxwell estimate \eqref{maxrotestone}, i.e.,
\begin{align}
\label{maxrotestonepde}
\forall\,E&\in\rc\cap\epsmo\rot\r&
\normepsom{E}&\leq\chmom\normmumoom{\rot E},&
\chmom&:=\cmtomepsmumo:=\ca.
\end{align}
Recently, the first author could show that, since $\om$ is convex, the upper bound
$$\chmom\leq\epso\,\muo\,\cpom$$
holds, see \cite{paulymaxconst0,paulymaxconst1,paulymaxconst2}.
Here, $\cpom$ denotes the Poincar\'e constant, i.e., the best constant in
\begin{align}
\label{poincare}
\forall\,u&\in\hob:=\ho\cap\rn^{\bot}&
\normom{u}&\leq\cpom\normom{\na u}
\end{align}
with the well known upper bound
$$\cpom\leq\frac{d_{\om}}{\pi},\qquad
d_{\om}:=\diam(\om),$$
see \cite{payneweinbergerpoincareconvex,bebendorfpoincareconvex}.
By the assumptions on $\eps$ and $\mu$ there exist $\epsu,\epso>0$
such that for all $E\in\Ltom$
$$\epsu^{-1}\normom{E}
\leq\normepsom{E}
\leq\epso\normom{E},\qquad
\epsu^{-1}\normepsom{E}
\leq\normom{\eps E}
\leq\epso\normepsom{E}.$$
We note $\normepsom{E}=\normom{\eps^{\nicefrac{1}{2}}E}$
and $\normepsom{\eps^{\nicefrac{1}{2}}E}=\normom{\eps E}$.
For the inverse $\epsmo$ we have the inverse estimates, i.e., for all $E\in\Ltom$
$$\epso^{-1}\normom{E}
\leq\normepsmoom{E}
\leq\epsu\normom{E},\quad
\epso^{-1}\normepsmoom{E}
\leq\normom{\epsmo E}
\leq\epsu\normepsmoom{E}.$$
We introduce the corresponding constants $\muu,\muo>0$ for $\mu$. 
We emphasize that the Helmholtz decompositions
\begin{align}
\label{helmhocom}
\Lteps
&=\na\hoc\oplus_{\eps}\epsmo\rot\r,&
\rc
&=\na\hoc\oplus_{\eps}(\rc\cap\epsmo\rot\r),\\
\label{helmhoom}
\Ltmu
&=\na\ho\oplus_{\mu}\mumo\rot\rc,&
\r
&=\na\ho\oplus_{\mu}(\r\cap\mumo\rot\rc)
\end{align}
hold since by the convexity of $\om$ 
$$\harmdeps=\{0\},\qquad
\harmnmu=\{0\},\qquad
\rot\r=\dz,\qquad
\rot\rc=\dcz.$$
Moreover, 
\begin{align*}
R(\pi)&=\pi\Lteps=\epsmo\rot\r,&
\pi\rc&=\rc\cap\epsmo\rot\r,\\
R(\pic)&=\pic\Ltmu=\mumo\rot\rc,&
\pic\r&=\r\cap\mumo\rot\rc
\end{align*}
and for $E\in\rc$ and $H\in\r$ we have
\begin{align}
\label{EHpirot}
\rot\pi E=\rot E,\qquad
\rot\pic H=\rot H.
\end{align}
Finally, we equip the Sobolev spaces $\hoc$ and $\hob$ with the norm $\normepsom{\na\,\cdot\,}$
as well as $\r$ and $\rc$ with the norm 
$\norm{\,\cdot\,}_{\r}:=\big(\normepsom{\,\cdot\,}^2+\normmumoom{\rot\,\cdot\,}^2\big)^{\nicefrac{1}{2}}$.

From now on, let us focus on the variational formulation of Theorem \ref{equivareqpde} (iii).

\subsection{A Saddle-Point Formulation}

For numerical purposes it is useful to split the condition
$\Eb\in\rc\cap\epsmo\rot\r$ into $\Eb\in\rc$ and $\eps\Eb\in\rot\r$.
Thanks to the vanishing Dirichlet fields we have 
$$\rot\r=\dz=(\na\hoc)^{\bot},$$
which is a nice and easy implementable condition.
Then, Theorem \ref{equivareqpde} (iii) is equivalent to: Find $\Eb\in\rc$ such that
\begin{align}
\label{vareqEneb3pde1}
\forall\,\Phi&\in\rc&
\scpmumoom{\rot\Eb}{\rot\Phi}
+\kamo\scpepsoms{\pioms\zets\Eb}{\pioms\zets\Phi}
&=\scpepsom{\zet\eh+\J}{\Phi}
-\scpom{\Hh}{\rot\Phi},\\
\label{vareqEneb3pde2}
\forall\,\varphi&\in\hoc&
\scpepsom{\Eb}{\na\varphi}
&=0.
\end{align}
Mixed formulations for this kind of systems
are well understood, see e.g. \cite[section 4.1]{giraultraviartbook}.
Let us define two continuous bilinear forms $a:\rc\times\rc\to\rn$, $b:\rc\times\hoc\to\rn$ 
and two continuous linear operators $\calA:\rc\to\rc'$, $\calB:\rc\to\hoc{}'$
as well as a continuous linear functional $f\in\rc'$ by
\begin{align*}
\forall\,\Psi,\Phi&\in\rc&
\calA\Psi(\Phi):=a(\Psi,\Phi)&:=\scpmumoom{\rot\Psi}{\rot\Phi}
+\kamo\scpepsoms{\pioms\zets\Psi}{\pioms\zets\Phi},\\
\forall\,\Psi&\in\rc,\varphi\in\hoc&
\calB\Psi(\varphi):=b(\Psi,\varphi)&:=\scpepsom{\Psi}{\na\varphi},\\
\forall\,\Phi&\in\rc&
f(\Phi)&:=\scpepsom{\zet\eh+\J}{\Phi}-\scpom{\Hh}{\rot\Phi}.
\end{align*}
Then, \eqref{vareqEneb3pde1}-\eqref{vareqEneb3pde2} read:
Find $\Eb\in\rc$, such that
\begin{align}
\label{vareqEneb5pde1}
\forall\,\Phi&\in\rc&
a(\Eb,\Phi)&=f(\Phi),\\
\label{vareqEneb5pde2}
\forall\,\varphi&\in\hoc&
b(\Eb,\varphi)&=0
\end{align}
or equivalently $\calA\Eb=f$ and $\calB\Eb=0$, i.e, $\Eb\in N(\calB)$ and $\calA\Eb=f$.
In matrix-notation this is
$$\zvec{\calA}{\calB}\Eb=\zvec{f}{0}.$$

\begin{theo}
\label{theobilinwithoutb}
The variational problem \eqref{vareqEneb5pde1}-\eqref{vareqEneb5pde2} 
is uniquely solvable. The unique solution is the adjoint state 
$\Eb\in\rc\cap\epsmo\dz$.
\end{theo}

\begin{proof}
\eqref{vareqEneb5pde2} is equivalent to $E\in\epsmo\dz=\epsmo\rot\r$.
Thus, unique solvability is clear by Theorem \ref{equivareqpde} (iii).
However, for convenience we present also another proof. For 
$$E\in N(\calB)=\rc\cap\epsmo\dz$$ 
we have by \eqref{maxrotestonepde}
\begin{align}
\label{acoerciveB}
a(E,E)\geq\normmumoom{\rot E}^2\geq(1+\chmom^2)^{-1}\norm{E}_{\r}^2,
\end{align}
i.e., $a$ is coercive over $N(\calB)$.
This shows uniqueness and that there exists a unique $E\in N(\calB)$, such that 
$$\forall\,\Phi\in N(\calB)\qquad a(E,\Phi)=f(\Phi)$$
holds. But then, this relation holds also for all $\Phi\in\rc$, i.e., 
\eqref{vareqEneb5pde1} holds, which proves existence. For this, let us decompose
$\rc\ni\Phi=\Phi_{\na}+\Phi_{0}\in\na\hoc\oplus_{\eps}N(\calB)$
by \eqref{helmhocom}. Then, by $\rot\Phi_{\na}=0$ and $\pioms\zets\Phi_{\na}=0$
since $\zets\Phi_{\na}\in\na\hooms$, see \eqref{helmoms}, 
as well as $\zet\eh+\J\in\epsmo\dz=R(\pi)$ by Lemma \ref{withoutpi},
we have
\begin{align*}
a(E,\Phi)
&=\scpmumoom{\rot E}{\rot\Phi}
+\kamo\scpepsoms{\pioms\zets E}{\pioms\zets\Phi}\\
&=\scpmumoom{\rot E}{\rot\Phi_{0}}
+\kamo\scpepsoms{\pioms\zets E}{\pioms\zets\Phi_{0}}
=a(E,\Phi_{0})=f(\Phi_{0})=f(\Phi).
\end{align*}
Theorem \ref{equivareqpde} shows $E=\Eb$.
\end{proof}

For numerical reasons we look at the following modification 
of \eqref{vareqEneb5pde1}-\eqref{vareqEneb5pde2},
defining a variational problem with a well known
saddle-point structure:
Find $(\Eb,\ub)\in\rc\times\hoc$, such that
\begin{align}
\label{vareqEneb6pde1}
\forall\,\Phi&\in\rc&
a(\Eb,\Phi)+b(\Phi,\ub)&=f(\Phi),\\
\label{vareqEneb6pde2}
\forall\,\varphi&\in\hoc&
b(\Eb,\varphi)&=0.
\end{align}
We note that $b(\Phi,\ub)=\calB\Phi(\ub)=\calB^{*}\ub(\Phi)$
with $\calB^{*}:\hoc\to\rc'$. So, \eqref{vareqEneb6pde1}-\eqref{vareqEneb6pde2}
may be written equivalently as $\calA\Eb+\calB^{*}\ub=f$
and $\calB\Eb=0$, i.e, $\Eb\in N(\calB)$ and $\calA\Eb+\calB^{*}\ub=f$.
In matrix-notation this is
$$\zmat{\calA}{\calB^{*}}{\calB}{0}\zvec{\Eb}{\ub}=\zvec{f}{0}.$$

\begin{lem}
\label{lembilinbzero}
For any solution $(E,u)\in\rc\times\hoc$ 
of \eqref{vareqEneb6pde1}-\eqref{vareqEneb6pde2}, i.e., of
\begin{align*}
\forall\,\Phi&\in\rc&
a(E,\Phi)+b(\Phi,u)&=f(\Phi),\\
\forall\,\varphi&\in\hoc&
b(E,\varphi)&=0,
\end{align*}
$u=0$ holds. 
\end{lem}

\begin{proof}
For $\varphi\in\ho$
we have $\pioms\zets\na\varphi=0$ 
as in the proof of the latter theorem
since $\zets\varphi\in\hooms$ 
and $\zets\na\varphi=\na\zets\varphi\in\na\hooms$.
Setting $\Phi:=\na u\in\rcz$, we get $\pioms\zets\Phi=0$ 
and hence $a(E,\Phi)=f(\Phi)=0$. 
But then $0=b(\Phi,u)=\normepsom{\na u}^2$,
yielding $u=0$.
\end{proof}

Now, it is clear that $(\Eb,0)$, where $\Eb$ is the unique solution 
of \eqref{vareqEneb5pde1}-\eqref{vareqEneb5pde2},
solves \eqref{vareqEneb6pde1}-\eqref{vareqEneb6pde2}.
On the other hand, any solution $(\Eb,\ub)$ of \eqref{vareqEneb6pde1}-\eqref{vareqEneb6pde2}
must satisfy $\ub=0$ and hence $\Eb$ must solve
\eqref{vareqEneb5pde1}-\eqref{vareqEneb5pde2}. This shows:

\begin{theo}
\label{bilinlem}
The variational formulation or saddle-point problem \eqref{vareqEneb6pde1}-\eqref{vareqEneb6pde2}
admits the unique solution $(\Eb,0)$.
\end{theo}

\begin{rem}
\label{bilinlemremproof}
Alternatively, we can prove the unique solvability 
of \eqref{vareqEneb6pde1}-\eqref{vareqEneb6pde2}
by a standard saddle-point technique,
e.g. by \cite[Corollary 4.1]{giraultraviartbook}.
We have already shown that $a$ is coercive over $N(\calB)=\rc\cap\epsmo\dz$, 
see \eqref{acoerciveB}. Moreover, as $\na\hoc=\rcz\subset\rc$, we have
for $0\neq\varphi\in\hoc$ with $\Phi:=\na\varphi\in\rcz$
$$\sup_{\Phi\in\rc}
\frac{b(\Phi,\varphi)}{\norm{\Phi}_{\r}\norm{\varphi}_{\hoc}}
\geq\frac{b(\na\varphi,\varphi)}{\norm{\na\varphi}_{\r}\normepsom{\na\varphi}}
=\frac{\normepsom{\na\varphi}^2}{\normepsom{\na\varphi}^2}
=1
\quad\impl\quad
\inf_{0\neq\varphi\in\hoc}
\sup_{\Phi\in\rc}
\frac{b(\Phi,\varphi)}{\norm{\Phi}_{\r}\norm{\varphi}_{\hoc}}
\geq1.$$
By Lemma \ref{lembilinbzero} we see that $\ub=0$.
\end{rem}

\subsection{A Double-Saddle-Point Formulation}

Now, we get rid of the unpleasant projector $\pioms$,
yielding another saddle-point structure.
For this, we assume for a moment that $\oms$ is additionally connected,
i.e., a bounded Lipschitz sub-domain of $\om$.
Let us decompose some $\xi\in\Ltepsoms$ by \eqref{helmoms}, i.e.,
$$\xi=-\na v+\epsmo\xi_{0}
\in\na\hooms\oplus_{\eps}\cE,\qquad
\cE=\epsmo\dczoms.$$
To compute $\xi_{0}$, we can choose $v\in\hoboms:=\hooms\cap\rn^{\bot}$ 
as the unique solution of the variational problem
\begin{align}
\label{piomscomp}
\forall\,\phi\in\hoboms\qquad 
\ka\,d(v,\phi):=\scpepsoms{\na v}{\na\phi}=-\scpepsoms{\xi}{\na\phi}.
\end{align}
Then, $\pioms\xi=\epsmo\xi_{0}=\xi+\na v$ and 
therefore for $E,\Phi\in\rc$ with $\xi:=\zets E$
\begin{align*}
a(E,\Phi)
&=\scpmumoom{\rot E}{\rot\Phi}
+\kamo\scpepsoms{\pioms\zets E}{\pioms\zets\Phi}
=\scpmumoom{\rot E}{\rot\Phi}
+\kamo\scpepsoms{\pioms\zets E}{\zets\Phi}\\
&=\underbrace{\scpmumoom{\rot E}{\rot\Phi}
+\kamo\scpepsoms{\zets E}{\zets\Phi}}_{\ds=:\tilde{a}(E,\Phi)}
+\underbrace{\kamo\scpepsoms{\na v}{\zets\Phi}}_{\ds=:c(\Phi,v)}.
\end{align*}
Hence, the saddle-point problem \eqref{vareqEneb6pde1}-\eqref{vareqEneb6pde2} 
can be written as the following variational double-saddle-point problem: 
Find $(\Eb,\ub,\vb)\in\rc\times\hoc\times\hoboms$, such that
\begin{align}
\label{vareqEneb7pde1}
\forall\,\Phi&\in\rc&
\tilde{a}(\Eb,\Phi)+b(\Phi,\ub)+c(\Phi,\vb)&=f(\Phi),\\
\label{vareqEneb7pde2}
\forall\,\varphi&\in\hoc&
b(\Eb,\varphi)&=0,\\
\label{vareqEneb7pde3}
\forall\,\phi&\in\hoboms&
c(\Eb,\phi)+d(\vb,\phi)&=0.
\end{align}
As before, now the continuous bilinear forms $\tilde{a}:\rc\times\rc\to\rn$ 
as well as $c:\rc\times\hoboms\to\rn$ and $d:\hoboms\times\hoboms\to\rn$ 
induce bounded linear operators $\tilde{\calA}:\rc\to\rc'$ as well as 
$\calC:\rc\to\hoboms'$ and $\calD:\hoboms\to\hoboms'$ by
\begin{align*}
\forall\,\Psi,\Phi&\in\rc&
\tilde{\calA}\Psi(\Phi):=\tilde{a}(\Psi,\Phi)&:=\scpmumoom{\rot\Psi}{\rot\Phi}
+\kamo\scpepsoms{\zets\Psi}{\zets\Phi},\\
\forall\,\Psi&\in\rc,\phi\in\hoboms&
\calC\Psi(\phi):=c(\Psi,\phi)&:=\kamo\scpepsoms{\zets\Psi}{\na\phi},\\
\forall\,\psi,\psi&\in\hoboms&
\calD\psi(\phi):=d(\psi,\phi)&:=\kamo\scpepsoms{\na\psi}{\na\phi}.
\end{align*}
We note that $c(\Phi,\vb)=\calC\Phi(\vb)=\calC^{*}\vb(\Phi)$
with $\calC^{*}:\hoboms\to\rc'$. So, \eqref{vareqEneb7pde1}-\eqref{vareqEneb7pde3}
may be written equivalently as $\tilde{\calA}\Eb+\calB^{*}\ub+\calC^{*}\vb=f$,
$\calB\Eb=0$ and $\calC\Eb+\calD\vb=0$, i.e, $\Eb\in N(\calB)$ 
and $\tilde{\calA}\Eb+\calB^{*}\ub+\calC^{*}\vb=f$, $\calC\Eb+\calD\vb=0$.
In matrix-notation this is
\begin{align}
\label{matrixdsp}
\dmat{\tilde{\calA}}{\calB^{*}}{\calC^{*}}{\calB}{0}{0}{\calC}{0}{\calD}
\dvec{\Eb}{\ub}{\vb}=\dvec{f}{0}{0}.
\end{align}
Note that we have formally
$$\Eb=(\tilde{\calA}-\calC^{*}\calD^{-1}\calC)^{-1}f$$
and formally in the strong sense
\begin{align*}
\tilde{\calA}
&\cong\rot_{\om}\mumo\overset{\circ}{\rot}_{\om}+\kamo\zet\eps\zets,&
\tilde{\calA}^{*}
&=\tilde{\calA},\\
\calB
&\cong-\div_{\om}\eps,&
\calB^{*}
&\cong\eps\overset{\circ}{\na}_{\om},\\
\calC
&\cong-\kamo\overset{\circ}{\div}_{\oms}\eps\zets,&
\calC^{*}
&\cong\kamo\zeta\eps\na_{\oms},\\
\calD
&\cong-\kamo\overset{\circ}{\div}_{\oms}\eps\na_{\oms},&
\calD^{*}
&=\tilde{\calD},&
f\cong\eps(\zet\eh+\J)-\rot\Hh.
\end{align*}
Here, the $\overset{\circ}{\,\cdot\,}$ and ${\,\cdot\,}_{\om}$, ${\,\cdot\,}_{\oms}$ 
indicate the boundary conditions and the domains, where the operators act, respectively.

\begin{theo}
\label{doublesaddleexist}
The variational formulation or double-saddle-point problem \eqref{vareqEneb7pde1}-\eqref{vareqEneb7pde3}
admits the unique solution $(\Eb,0,\vb)$ with $\na\vb=(\pioms-1)\zets\Eb$.
Moreover, $\eb=\eh-\kamo\pioms\zets\Eb=\eh-\kamo(\zets\Eb+\na\vb)$
defines the optimal control. 
\end{theo}

\begin{proof}
Since $\pioms\zets\Eb=\zets\Eb+\na\vb$, if and only if $\vb\in\hoboms$ and
$$\forall\,\phi\in\hoboms\qquad
c(\Eb,\phi)+d(\vb,\phi)=0,$$
we have
\begin{align*}
\forall\,\Phi&\in\rc&
a(\Eb,\Phi)+b(\Phi,\ub)&=f(\Phi),
\intertext{if and only if $\pioms\zets\Eb=\zets\Eb+\na\vb$ and}
\forall\,\Phi&\in\rc&
\tilde{a}(\Eb,\Phi)+b(\Phi,\ub)+c(\Phi,\vb)&=f(\Phi),
\intertext{if and only if $\vb\in\hoboms$ and}
\forall\,\Phi&\in\rc&
\tilde{a}(\Eb,\Phi)+b(\Phi,\ub)+c(\Phi,\vb)&=f(\Phi),\\
\forall\,\phi&\in\hoboms&
c(\Eb,\phi)+d(\vb,\phi)&=0.
\end{align*}
Hence, the unique solvability follows immediately by Theorem \ref{bilinlem}.
\end{proof}

\begin{rem}
\label{dspremproof}
As in Remark \ref{bilinlemremproof} we give an alternative proof
using the double-saddle-point structure of the problem.
We rearrange the equations and variables in \eqref{matrixdsp} equivalently as
$$\dmat{\tilde{\calA}}{\calC^{*}}{\calB^{*}}{\calC}{\calD}{0}{\calB}{0}{0}
\dvec{\Eb}{\vb}{\ub}=\dvec{f}{0}{0}$$
and obtain
$$\zmat{\hat{\calA}}{\hat{\calB}^{*}}{\hat{\calB}}{0}
\zvec{(\Eb,\vb)}{\ub}=\zvec{\hat{f}}{0},\quad
\hat{\calA}:=\zmat{\tilde{\calA}}{\calC^{*}}{\calC}{\calD},\quad
\hat{\calB}:=[\calB\;0],\quad
\hat{\calB}^{*}=\zvec{\calB^{*}}{0},\quad
\hat{f}=\zvec{f}{0}.$$
Now, $\hat{\calA}:\rc\times\hoboms\to\big(\rc\times\hoboms\big)'$,
$\hat{\calB}:\rc\times\hoboms\to\hoc{}'$,
$\hat{\calB}^{*}:\hoc\to\big(\rc\times\hoboms\big)'$
and $\hat{f}\in\big(\rc\times\hoboms\big)'$.
For bilinear forms this means:
Find $\big((\Eb,\vb),\ub\big)\in\big(\rc\times\hoboms\big)\times\hoc$, such that
\begin{align}
\label{vareqEneb8pde1}
\forall\,(\Phi,\phi)&\in\rc\times\hoboms&
\hat{a}\big((\Eb,\vb),(\Phi,\phi)\big)
+\hat{b}\big((\Phi,\phi),\ub\big)
&=\hat{f}\big((\Phi,\phi)\big),\\
\label{vareqEneb8pde2}
\forall\,\varphi&\in\hoc&
\hat{b}\big((\Eb,\vb),\varphi\big)
&=0,
\end{align}
where for $(\Psi,\psi),(\Phi,\phi)\in\rc\times\hoboms$ and $\varphi\in\hoc$
\begin{align*}
\hat{\calA}(\Psi,\psi)\big((\Phi,\phi)\big)
=\hat{a}\big((\Psi,\psi),(\Phi,\phi)\big)
&:=\tilde{a}(\Psi,\Phi)
+c(\Phi,\psi)
+c(\Psi,\phi)
+d(\psi,\phi),\\
\hat{\calB}^{*}\varphi(\Psi,\psi)
=\hat{\calB}(\Psi,\psi)(\varphi)
=\hat{b}\big((\Psi,\psi),\varphi\big)
&:=b(\Psi,\varphi),\\
\hat{f}\big((\Phi,\phi)\big)
&:=f(\Phi).
\end{align*}
Now, we can prove the unique solvability 
of \eqref{vareqEneb8pde1}-\eqref{vareqEneb8pde2}
by the same standard saddle-point technique from \cite[Corollary 4.1]{giraultraviartbook}.
As $a$ is coercive over $N(\calB)=\rc\cap\epsmo\dz$, see \eqref{acoerciveB}, 
so is $\hat{a}$ over the kernel 
$N(\hat{\calB})=N(\calB)\times\hoboms=(\rc\cap\epsmo\dz)\times\hoboms$.
More precisely, for all $(E,v)\in N(\hat{\calB})$ and $\delta\in(0,1)$
\begin{align*}
\hat{a}\big((E,v),(E,v)\big)
&=\tilde{a}\big((E,v),(E,v)\big)
+2c(E,v)
+d(v,v)\\
&=\normmumoom{\rot E}^2
+\kamo\normepsoms{\zets E}^2
+2\kamo\scpepsoms{\zets E}{\na v}
+\kamo\normepsoms{\na v}^2\\
&=\normmumoom{\rot E}^2
+\kamo\normepsoms{\zets E+\na v}^2\\
&\geq(1+\chmom^2)^{-1}\norm{E}_{\r}^2
+\delta\kamo\normepsoms{\zets E+\na v}^2\\
&\geq\frac{1}{1+\chmom^2}\normmumoom{\rot E}^2
+\frac{1}{1+\chmom^2}\normepsom{E}^2
-\frac{\delta}{\ka}\normepsoms{\zets E}^2
+\frac{\delta}{2\ka}\normepsoms{\na v}^2\\
&\geq\frac{1}{1+\chmom^2}\normmumoom{\rot E}^2
+\big(\frac{1}{1+\chmom^2}-\frac{\delta}{\ka}\big)\normepsom{E}^2
+\frac{\delta}{2\ka}\normepsoms{\na v}^2.
\end{align*}
Hence, $\alpha\,\hat{a}\big((E,v),(E,v)\big)\geq\norm{E}_{\r}^2+\norm{v}_{\hoboms}^2=\norm{(E,v)}_{\r\times\hoboms}^2$
for $\delta$ sufficiently small with some $\alpha>0$. Then, as before,
for $0\neq\varphi\in\hoc$ with $\Phi:=\na\varphi\in\rcz$ and now also $\phi:=0$
\begin{align*}
\sup_{(\Phi,\phi)\in\rc\times\hoboms}
\frac{\hat{b}\big((\Phi,\phi),\varphi\big)}{\norm{(\Phi,\phi)}_{\r\times\hoboms}\norm{\varphi}_{\hoc}}
&=\sup_{(\Phi,\phi)\in\rc\times\hoboms}
\frac{b(\Phi,\varphi)}{\norm{(\Phi,\phi)}_{\r\times\hoboms}\norm{\varphi}_{\hoc}}\\
&\geq\frac{b(\na\varphi,\varphi)}{\norm{\na\varphi}_{\r}\normepsom{\na\varphi}}
=\frac{\normepsom{\na\varphi}^2}{\normepsom{\na\varphi}^2}
=1
\end{align*}
and thus
$$\inf_{0\neq\varphi\in\hoc}
\sup_{(\Phi,\phi)\in\rc\times\hoboms}
\frac{\hat{b}\big((\Phi,\phi),\varphi\big)}{\norm{(\Phi,\phi)}_{\r\times\hoboms}\norm{\varphi}_{\hoc}}
\geq1.$$
Therefore, \eqref{vareqEneb8pde1}-\eqref{vareqEneb8pde2} is uniquely solvable.
This is equivalent to \eqref{vareqEneb7pde1}-\eqref{vareqEneb7pde3}.
Moreover by \eqref{vareqEneb7pde3} we see $\na\vb=(\pioms-1)\zets\Eb$.
Hence, $(\Eb,\ub)$ is the unique solution of \eqref{vareqEneb6pde1}-\eqref{vareqEneb6pde2}
and Lemma \ref{lembilinbzero} shows $\ub=0$.
\end{rem}

\begin{rem}
\label{remomsconn}
We emphasize that \eqref{vareqEneb7pde3} holds for all $\phi\in\hooms$ as well,
since only $\na\phi$ and $\na\vb$ occur. Hence, we can also search for
$\vb\in\hooms$, where in this case $\vb$ is uniquely determined up to constants.
This shows also, that we can skip again the additional assumption of a connected $\oms$.
Then, $\vb$ may be uniquely defined just up to constants in the connected subdomains of $\oms$,
but this does not change the uniqueness of the orthogonal Helmholtz projector
$\pioms\zets\Eb=\zets\Eb+\na\vb$.
\end{rem}

Finally, we write down the double-saddle-point problem 
\eqref{vareqEneb7pde1}-\eqref{vareqEneb7pde3} in a more explicit form:
Find $(\Eb,\ub,\vb)\in\rc\times\hoc\times\hooms$, such that
\begin{align}
\label{vareqEneb9pde1}
\forall\,\Phi&\in\rc&
\scpmumoom{\rot\Eb}{\rot\Phi}
+\kamo\scpepsoms{\zets\Eb}{\zets\Phi}\quad\\
\nonumber
&&+\scpepsom{\Phi}{\na\ub}
+\kamo\scpepsoms{\zets\Phi}{\na\vb}
&=\scpepsom{\zet\eh+\J}{\Phi}
-\scpom{\Hh}{\rot\Phi},\\
\label{vareqEneb9pde2}
\forall\,\varphi&\in\hoc&
\scpepsom{\Eb}{\na\varphi}
&=0,\\
\label{vareqEneb9pde3}
\forall\,\phi&\in\hooms&
\kamo\scpepsoms{\zets\Eb}{\na\phi}
+\kamo\scpepsoms{\na\vb}{\na\phi}
&=0.
\end{align}
Or altogether:
Find $(\Eb,\ub,\vb)\in\rc\times\hoc\times\hooms$, such that for all
$(\Phi,\varphi,\phi)\in\rc\times\hoc\times\hooms$
\begin{align}
\label{vareqEneb10pde}
&\scpmumoom{\rot\Eb}{\rot\Phi}
+\kamo\scpepsoms{\zets\Eb}{\zets\Phi}
+\scpepsom{\Phi}{\na\ub}
+\kamo\scpepsoms{\zets\Phi}{\na\vb}\\
&\quad+\scpepsom{\Eb}{\na\varphi}
+\kamo\scpepsoms{\zets\Eb}{\na\phi}
+\kamo\scpepsoms{\na\vb}{\na\phi}
+\scpom{\Hh}{\rot\Phi}
-\scpepsom{\zet\eh+\J}{\Phi}
=0.\non
\end{align}
The unique optimal control is 
$$\eb=\eh-\kamo\pioms\zets\Eb=\eh-\kamo(\zets\Eb+\na\vb)\in\epsmo\dczoms=\cE.$$
Note that $\zet\eb\in\epsmo\dcz$ 
and that $\vb\in\hooms$ is only unique up to constants in connected
parts of $\oms$.

\section{Functional A Posteriori Error Analysis}

We will derive functional a posteriori error estimates in the spirit 
of Repin \cite{repinbookone,paulyrepinmaxst}.
Especially, we are interested in estimating the error of the optimal control $\eb-\et$.

Let $\Et\in\rc$ and $\vt\in\hooms$. Then
\begin{equation}
\label{defforj}
\Et\in\rc,\quad
\et:=\eh-\kamo(\zets\Et+\na\vt)\in\Ltepsoms,\quad
\Hti:=\mumo\rot\Et+\Hh\in\mumo\dcz
\end{equation}
may be considered as approximations of the adjoint state, the optimal control and the state
$$\Eb\in\rc\cap\epsmo\dz,\quad
\eb\in\epsmo\dczoms,\quad
\Hb\in\r\cap\mumo\dcz,$$ 
respectively. We note
\begin{align*}
\eb-\et&=\kamo(\zets\Et+\na\vt-\pioms\zets\Eb)
=\kamo\big(\zets(\Et-\Eb)+\na(\vt-\vb)\big)\in\roms,\\
\Hb-\Hti&=\mumo\rot(\Eb-\Et)\in\mumo\dcz
\intertext{and hence}
\ka\rot(\eb-\et)
&=\rot\zets(\Et-\Eb)
=\zets\rot(\Et-\Eb)
=\mu\zets(\Hti-\Hb)
\in\rot\roms.
\end{align*}
If $\eh\in\roms$, then $\eb\in\roms\cap\epsmo\dczoms$
and $\et\in\roms$.

First, we will focus on the variational formulation 
\eqref{vareqEneb5pde1}, i.e., \eqref{vareqEneb3pde1}.
We note, that 
$$\scpom{\Hh}{\rot\Phi}
=\scpom{\rot\Hh}{\Phi}$$
holds for $\Phi\in\rc$ and $\Hh\in\r$,
giving two options for putting $\Hh$ in our estimates
depending on its regularity.

\subsection{Upper Bounds}

For all $\Phi\in\rc$ and all $\Psi\in\r$ we have by \eqref{vareqEneb3pde1}
\begin{align*}
&\qquad\scpmumoom{\rot(\Eb-\Et)}{\rot\Phi}
+\kamo\scpepsoms{\pioms\zets(\Eb-\Et)}{\pioms\zets\Phi}\\
&=-\scpmumoom{\mu\Hh+\rot\Et}{\rot\Phi}
+\scpepsoms{\eh-\kamo\pioms\zets\Et}{\zets\Phi}
+\scpepsom{\J}{\Phi}\\
&=-\scpmumoom{\mu\Hti}{\rot\Phi}
+\scpepsom{\zet\eh+\J-\kamo\zet\pioms\zets\Et}{\Phi}\\
&=\scpmumoom{\mu(\Psi-\Hti)}{\rot\Phi}
+\scpepsom{\zet\eh+\J-\kamo\zet\pioms\zets\Et-\epsmo\rot\Psi}{\Phi}.
\end{align*}
Since $\J,\epsmo\rot\Psi\in\epsmo\rot\r=R(\pi)$ as well as
$\zet\pioms\zets\Et=\pi\zet\pioms\zets\Et$ and 
$\zet\eh=\zet\pioms\eh=\pi\zet\pioms\eh=\pi\zet\eh$
by Lemma \ref{withoutpi}, we see 
$$R(\pi)\ni\zet\eh+\J-\kamo\zet\pioms\zets\Et-\epsmo\rot\Psi
=\pi(\zet\eh+\J-\kamo\zet\pioms\zets\Et-\epsmo\rot\Psi).$$
Thus,
\begin{align}
\label{aposterrorI1}
\begin{split}
&\qquad\scpmumoom{\rot(\Eb-\Et)}{\rot\Phi}
+\kamo\scpepsoms{\pioms\zets(\Eb-\Et)}{\pioms\zets\Phi}\\
&=\scpmumoom{\mu(\Psi-\Hti)}{\rot\Phi}
+\scpepsom{\zet\eh+\J-\kamo\zet\pioms\zets\Et-\epsmo\rot\Psi}{\pi\Phi}.
\end{split}
\end{align}
As $\pi\Phi\in\rc\cap\epsmo\rot\r$ with $\rot\pi\Phi=\rot\Phi$
by \eqref{EHpirot} we get by \eqref{maxrotestonepde}
\begin{align}
\label{maxrotesttwopde}
\normepsom{\pi\Phi}
\leq\chmom\normmumoom{\rot\Phi}.
\end{align}
Therefore, by \eqref{aposterrorI1} 
\begin{align}
\label{aposterrorI2}
\scpmumoom{\rot(\Eb-\Et)}{\rot\Phi}
+\kamo\scpepsoms{\pioms\zets(\Eb-\Et)}{\pioms\zets\Phi}
&\leq\cM_{+,\rot,\pioms}(\Et,\Hti;\Psi)
\normmumoom{\rot\Phi},
\end{align}
where
\begin{align*}
\cM_{+,\rot,\pioms}(\Et,\Hti;\Psi)
&:=\normmuom{\Hti-\Psi}
+\chmom\normepsom{\zet\eh+\J-\kamo\zet\pioms\zets\Et-\epsmo\rot\Psi}.
\intertext{Note that $\cM_{+,\rot,\pioms}$ can be replaced by}
\cMt_{+,\rot,\pioms}(\Et;\Psi)
&:=\normmumoom{\rot\Et-\mu\Psi}
+\chmom\normepsom{\zet\eh+\J-\kamo\zet\pioms\zets\Et-\epsmo\rot(\Psi+\Hh)},
\end{align*}
if $\Hh\in\r$, since $\epsmo\rot\Hh\in R(\pi)$.
Inserting $\Phi:=\Eb-\Et\in\rc$ into \eqref{aposterrorI2} yields for all $\Psi\in\r$
\begin{align}
\label{aposterrorI3}
\dnormrot{\Eb-\Et}
\leq\cM_{+,\rot,\pioms}(\Et,\Hti;\Psi),
\end{align}
where we define $\dnormrot{\,\cdot\,}$ by
\begin{align*}
\dnormrot{\Phi}^2
&:=\normmumoom{\rot\Phi}^2
+\frac{1}{\ka}\normepsoms{\pioms\zets\Phi}^2,\qquad
\Phi\in\r.
\end{align*}

To estimate the possibly non-solenoidal part of the error we decompose
$\Et$ by the Helmholtz decomposition \eqref{helmhocom}
$$\Et=\na\tilde{\varphi}+\pi\Et\in\na\hoc\oplus_{\eps}(\rc\cap\epsmo\rot\r),\qquad
\rot\pi\Et=\rot\Et.$$
Then, for all $\Phi\in\epsmo\d$
$$\normepsom{\na\tilde{\varphi}}^2
=\scpepsom{\Et}{\na\tilde{\varphi}}
=\scpepsom{\Et-\Phi}{\na\tilde{\varphi}}
-\scpom{\div\eps\Phi}{\tilde{\varphi}}
\leq\cM_{+,\div}(\Et;\Phi)
\normepsom{\na\tilde{\varphi}}$$
and hence 
$$\normepsom{\na\tilde{\varphi}}\leq\cM_{+,\div}(\Et;\Phi),\quad
\cM_{+,\div}(\Et;\Phi)
:=\normepsom{\Et-\Phi}
+\chpom\normom{\div\eps\Phi}.$$
Here, $\chpom:=\cpcomeps$ is the Poincar\'e constant in the Poincar\'e inequality
\begin{align}
\label{poincarehoc}
\forall\,\varphi\in\hoc\qquad
\normom{\varphi}\leq\chpom\normepsom{\na\varphi}
\end{align}
and we emphasize 
$$\chpom\leq\epsu\cpcom,\qquad
\cpcom<\cpom\leq\frac{d_{\om}}{\pi}.$$
As $\Eb$ already belongs to $\rc\cap\epsmo\rot\r$ we have
$\Eb-\Et=\pi(\Eb-\Et)-\na\tilde{\varphi}$
and obtain by orthogonality and by \eqref{EHpirot}, \eqref{maxrotesttwopde} 
for all $\Psi\in\r$ and all $\Phi\in\epsmo\d$
\begin{align*}
\normepsom{\Eb-\Et}^2
&=\normepsom{\na\tilde{\varphi}}^2
+\normepsom{\pi(\Eb-\Et)}^2
\leq\cM_{+,\div}^2(\Et;\Phi)
+\chmom^2\normmumoom{\rot(\Eb-\Et)}^2,\\
\dnorm{\Eb-\Et}^2
&\leq\cM_{+,\div}^2(\Et;\Phi)
+\chmom^2\dnormrot{\Eb-\Et}^2,
\end{align*}
where $\dnorm{\,\cdot\,}$ is defined by
$$\dnorm{\Phi}^2
:=\normepsom{\Phi}^2
+\frac{\chmom^2}{\ka}\normepsoms{\pioms\zets\Phi}^2,\qquad
\Phi\in\Lteps.$$
Let us underline the norm equivalence for $\Phi\in\r$
\begin{align*}
\norm{\Phi}_{\r}^2
&\leq\dnorm{\Phi}_{\r}^2
=\normepsom{\Phi}^2
+\normmumoom{\rot\Phi}^2
+\frac{1+\chmom^2}{\ka}\normepsoms{\pioms\zets\Phi}^2\\
&\leq\big(1+\frac{1+\chmom^2}{\ka}\big)\normepsom{\Phi}^2
+\normmumoom{\rot\Phi}^2
\leq\big(1+\frac{1+\chmom^2}{\ka}\big)\norm{\Phi}_{\r}^2,
\end{align*}
where $\dnorm{\,\cdot\,}_{\r}$ is defined by
$$\dnorm{\Phi}_{\r}^2
:=\dnorm{\Phi}^2
+\dnormrot{\Phi}^2,\qquad
\Phi\in\r,$$
i.e.,
$\ds\dnorm{\Phi}_{\r}^2
=\normepsom{\Phi}^2
+\normmumoom{\rot\Phi}^2
+\frac{1+\chmom^2}{\ka}\normepsoms{\pioms\zets\Phi}^2$.

\begin{lem}
\label{apostlemupperboundI}
Let $\Et\in\rc$. Then, for all $\Phi\in\epsmo\d$ and all $\Psi\in\r$
\begin{align*}
\dnorm{\Eb-\Et}^2
&\leq\chmom^2\dnormrot{\Eb-\Et}^2
+\cM_{+,\div}^2(\Et;\Phi),\\
\dnorm{\Eb-\Et}_{\r}^2
&\leq(1+\chmom^2)\dnormrot{\Eb-\Et}^2
+\cM_{+,\div}^2(\Et;\Phi),\\
\dnormrot{\Eb-\Et}
&\leq\cM_{+,\rot,\pioms}(\Et,\Hti;\Psi),
\end{align*}
where
\begin{align*}
\cM_{+,\rot,\pioms}(\Et,\Hti;\Psi)
&=\normmuom{\Hti-\Psi}
+\chmom\normepsom{\zet\eh+\J-\kamo\zet\pioms\zets\Et-\epsmo\rot\Psi},\\
\cM_{+,\div}(\Et;\Phi)
&=\normepsom{\Et-\Phi}
+\chpom\normom{\div\eps\Phi}
\end{align*}
and $\cM_{+,\rot,\pioms}$ can be replaced by $\cMt_{+,\rot,\pioms}$,
if $\Hh\in\r$.
\end{lem}

\begin{rem}
\label{apostremupperboundI}
We note that by the convexity of $\om$
all appearing constants have easily computable upper bounds, i.e.,
$$\chpom\leq\epsu\cpcom,\qquad
\chmom\leq\epso\,\muo\,\cpom,\qquad
\cpcom<\cpom\leq\frac{d_{\om}}{\pi}.$$
\end{rem}

Setting $\Phi:=\Eb\in\epsmo\dz$ we get
$$\cM_{+,\div}(\Et;\Eb)
=\normepsom{\Eb-\Et}.$$
For $\Psi:=\Hb\in\r$ we see 
$\mu\Hb=\rot\Eb+\mu\Hh$ and
$\epsmo\rot\Hb=\zet\eh+\J-\kamo\zet\pioms\zets\Eb$
and thus
$$\cM_{+,\rot,\pioms}(\Et,\Hti;\Hb)
=\normmuom{\Hb-\Hti}
+\frac{\chmom}{\ka}\normepsoms{\pioms\zets(\Eb-\Et)}
\leq c_{\ka}\dnormrot{\Eb-\Et}$$
by $\mu(\Hb-\Hti)=\rot(\Eb-\Et)$ and with
$$c_{\ka}:=\big(1+\frac{\chmom^2}{\ka}\big)^{\nicefrac{1}{2}}.$$
For $\Hh\in\r$ and defining $\Psi:=\Hb-\Hh\in\r$ we see 
$$\cMt_{+,\rot,\pioms}(\Et,\Hb-\Hh)
=\cM_{+,\rot,\pioms}(\Et,\Hti;\Hb).$$

\begin{rem}
\label{apostremupperboundIsharp}
In Lemma \ref{apostlemupperboundI}, the upper bounds are equivalent 
to the respective norms of the error. More precisely, it holds
\begin{align*}
\dnormrot{\Eb-\Et}
&\leq\inf_{\Psi\in\r}\cM_{+,\rot,\pioms}(\Et,\Hti;\Psi)
\leq\cM_{+,\rot,\pioms}(\Et,\Hti;\Hb)
\leq c_{\ka}\dnormrot{\Eb-\Et},\\
\dnorm{\Eb-\Et}_{\r}^2
&\leq(1+\chmom^2)\inf_{\Psi\in\r}\cM_{+,\rot,\pioms}^2(\Et,\Hti;\Psi)
+\inf_{\Phi\in\epsmo\d}\cM_{+,\div}^2(\Et;\Phi)\\
&\leq(1+\chmom^2)\cM_{+,\rot,\pioms}^2(\Et,\Hti;\Hb)
+\cM_{+,\div}^2(\Et;\Eb)\\
&\leq c_{\ka}^2(1+\chmom^2)\dnormrot{\Eb-\Et}^2
+\normepsom{\Eb-\Et}^2
\leq c_{\ka}^2(1+\chmom^2)\dnorm{\Eb-\Et}_{\r}^2.
\end{align*}
If $\Hh\in\r$, the majorant $\ds\inf_{\Psi\in\r}\cM_{+,\rot,\pioms}(\Et,\Hti;\Psi)$
can be replaced by $\ds\inf_{\Psi\in\r}\cMt_{+,\rot,\pioms}(\Et;\Psi)$ and the terms
$\cM_{+,\rot,\pioms}(\Et,\Hti;\Hb)$ by $\cMt_{+,\rot,\pioms}(\Et,\Hb-\Hh)$.
\end{rem}

In Lemma \ref{apostlemupperboundI}, the upper bounds are explicitly computable
except of the unpleasant projector $\pioms$.
Moreover, so far we can estimate only the terms
$$\Eb-\Et,\qquad
\rot(\Eb-\Et),\qquad
\pioms\zets(\Eb-\Et),$$
but we are manly interested in estimating 
the error of the optimal control $\eb-\et$, where
$$\ka(\eb-\et)
=-\pioms\zets\Eb+\zets\Et+\na\vt
=\zets(\Et-\Eb)+\na(\vt-\vb).$$
We note
\begin{align}
\label{estnavbvt1}
\normepsoms{\na(\vb-\vt)}
&\leq\ka\normepsoms{\eb-\et}
+\normepsoms{\zets(\Eb-\Et)}.
\end{align}
To attack these problems, we note that the projector $\pioms$
is computed by \eqref{piomscomp} as follows:
For $\xi\in\Ltepsoms$ we solve the weighted Neumann Laplace problem
$$\forall\,\phi\in\hoboms\qquad 
\scpepsoms{\na v}{\na\phi}=-\scpepsoms{\xi}{\na\phi}$$
with $v=v_{\xi}\in\hoboms$. Then, $\pioms\xi=\xi+\na v$.
Now, for $\vt\in\hooms$ as well as for all $\phi\in\hooms$
and all $\Upsilon\in\epsmo\dcoms$ we have
$$\scpepsoms{\na(v-\vt)}{\na\phi}
=\scpepsoms{\Upsilon-\xi-\na\vt}{\na\phi_{\bot}}
+\scpoms{\div\eps\Upsilon}{\phi_{\bot}}
\leq\big(\normepsoms{\Upsilon-\xi-\na\vt}
+\chpoms\normoms{\div\eps\Upsilon}\big)
\normepsoms{\na\phi},$$
where $\phi_{\bot}\in\hoboms$ with $\na\phi=\na\phi_{\bot}$.
Here, $\chpoms:=\cpomseps$ is the Poincar\'e constant in the Poincar\'e inequality
\begin{align}
\label{poincarehocoms}
\forall\,\phi\in\hoboms\qquad
\normoms{\phi}\leq\chpom\normepsoms{\na\phi}
\end{align}
and we note
$$\chpoms\leq\epsu\cpoms,$$
where $\cpoms\leq d_{\oms}/\pi$ if $\oms$ is convex.
Hence, putting $\phi:=v-\vt$ gives
$$\normepsoms{\na(v-\vt)}
\leq\normepsoms{\xi+\na\vt-\Upsilon}
+\chpoms\normoms{\div\eps\Upsilon}.$$
Especially for $\xi:=\zets\Et$ with $\pioms\zets\Et=\zets\Et+\na v$ we obtain immediately
\begin{align*}
\ka(\et-\eb)
&=\pioms\zets(\Eb-\Et)+\na(v-\vt),\\
\ka^2\normepsoms{\eb-\et}^2
&=\normepsoms{\pioms\zets(\Eb-\Et)}^2
+\normepsoms{\na(v-\vt)}^2,\\
\normepsoms{\na(v-\vt)}
&\leq\normepsoms{\zets\Et+\na\vt-\Upsilon}
+\chpoms\normoms{\div\eps\Upsilon}
=:\cM_{+,\pioms}(\Et,\vt;\Upsilon).
\end{align*}
We remark $\pioms\zets\Eb=\zets\Eb+\na\vb$ giving
\begin{align*}
\zets(\Eb-\Et)
&=\pioms\zets(\Eb-\Et)
+\na(v-\vb),\\
\normepsoms{\zets(\Eb-\Et)}^2
&=\normepsoms{\pioms\zets(\Eb-\Et)}^2
+\normepsoms{\na(\vb-v)}^2.
\end{align*}
This shows
\begin{align*}
\normepsoms{\na(v-\vt)},
\normepsoms{\pioms\zets(\Eb-\Et)}
&\leq\ka\normepsoms{\eb-\et},\\
\normepsoms{\na(\vb-v)},
\normepsoms{\pioms\zets(\Eb-\Et)}
&\leq\normepsoms{\zets(\Eb-\Et)}
\end{align*}
and thus \eqref{estnavbvt1} follows again.
We note that as
$$\ka\rot(\eb-\et)
=\zets\rot(\Et-\Eb)
=\mu\zets(\Hti-\Hb)$$
and hence
$$\ka\normmumooms{\rot(\eb-\et)}
=\normmumooms{\zets\rot(\Eb-\Et)}
=\normmuoms{\zets(\Hb-\Hti)}$$
we can even estimate $\eb-\et$ in $\roms$.
More precisely,
\begin{align*}
\ka\normepsoms{\eb-\et}^2
+\ka^2\normmumooms{\rot(\eb-\et)}^2
&\leq\ka\normepsoms{\eb-\et}^2
+\normmuom{\Hb-\Hti}^2\\
&=\kamo\normepsoms{\pioms\zets(\Eb-\Et)}^2
+\kamo\normepsoms{\na(v-\vt)}^2
+\normmumoom{\rot(\Eb-\Et)}^2\\
&\leq\dnormrot{\Eb-\Et}^2
+\kamo\cM_{+,\pioms}^2(\Et,\vt;\Upsilon).
\end{align*}

Next, we find a computable upper bound for the term
$\normepsom{\zet\eh+\J-\kamo\zet\pioms\zets\Et-\epsmo\rot\Psi}$ 
in the majorant $\cM_{+,\rot,\pioms}(\Et,\Hti;\Psi)$,
simply by inserting $\pioms\zets\Et=\zets\Et+\na\vt+\na(v-\vt)$, yielding
\begin{align*}
\normepsom{\zet\eh+\J-\kamo\zet\pioms\zets\Et-\epsmo\rot\Psi}
&\leq\normepsom{\zet\eh+\J-\kamo\zet(\zets\Et+\na\vt)-\epsmo\rot\Psi}
+\kamo\normepsoms{\na(v-\vt)}\\
&\leq\normepsom{\zet\et+\J-\epsmo\rot\Psi}
+\kamo\cM_{+,\pioms}(\Et,\vt;\Upsilon).
\end{align*}
Putting all together shows:

\begin{lem}
\label{apostlemupperboundII}
Let $\Et\in\rc$ and $\vt\in\hooms$. Furthermore, let $\et:=\eh-\kamo(\zets\Et+\na\vt)\in\Ltepsoms$ and
$\Hti:=\mumo\rot\Et+\Hh\in\mumo\dcz.$
Then, for all $\Phi\in\epsmo\d$, for all $\Psi\in\r$
and for all $\Upsilon\in\epsmo\dcoms$
\begin{align*}
\normepsoms{\na(\vb-\vt)}
&\leq\normepsoms{\zets(\Eb-\Et)}
+\min\big\{\ka\normepsoms{\eb-\et},
\cM_{+,\pioms}(\Et,\vt;\Upsilon)\big\},\\
\ka\normmumooms{\rot(\eb-\et)}
&=\normmuoms{\zets(\Hb-\Hti)}
\leq\normmuom{\Hb-\Hti}
=\normmumoom{\rot(\Eb-\Et)},\\
\ka\normepsoms{\eb-\et}^2
+\normmuom{\Hb-\Hti}^2
&\leq\dnormrot{\Eb-\Et}^2
+\kamo\cM_{+,\pioms}^2(\Et,\vt;\Upsilon),\\
\dnorm{\Eb-\Et}^2
&\leq\chmom^2\dnormrot{\Eb-\Et}^2
+\cM_{+,\div}^2(\Et;\Phi),\\
\dnorm{\Eb-\Et}_{\r}^2
&\leq(1+\chmom^2)\dnormrot{\Eb-\Et}^2
+\cM_{+,\div}^2(\Et;\Phi),\\
\dnormrot{\Eb-\Et}
\leq\cM_{+,\rot,\pioms}(\Et,\Hti;\Psi)
&\leq\cM_{+,\rot}(\Hti,\et;\Psi)
+\kamo\chmom\cM_{+,\pioms}(\Et,\vt;\Upsilon),
\end{align*}
where
\begin{align*}
\cM_{+,\rot}(\Hti,\et;\Psi)
&:=\normmuom{\Hti-\Psi}
+\chmom\normepsom{\zet\et+\J-\epsmo\rot\Psi},\\
\cM_{+,\div}(\Et;\Phi)
&\;=\normepsom{\Et-\Phi}
+\chpom\normom{\div\eps\Phi},\\
\cM_{+,\pioms}(\Et,\vt;\Upsilon)
&\;=\normepsoms{\zets\Et+\na\vt-\Upsilon}
+\chpoms\normoms{\div\eps\Upsilon}.
\intertext{If $\Hh\in\r$, $\cM_{+,\rot}$ can be replaced by $\cMt_{+,\rot}$ with}
\cMt_{+,\rot}(\Et,\et;\Psi)
&:=\normmumoom{\rot\Et-\mu\Psi}
+\chmom\normepsom{\zet\et+\J-\epsmo\rot(\Psi+\Hh)}.
\end{align*}
\end{lem}

For $\Upsilon:=\pioms\zets\Eb=\zets\Eb+\na\vb\in\epsmo\dczoms$ we have
$$\cM_{+,\pioms}(\Et,\vt;\pioms\zets\Eb)
=\ka\normepsoms{\eb-\et}
\leq\normepsoms{\zets(\Eb-\Et)}
+\normepsoms{\na(\vb-\vt)}.$$
For $\Psi:=\Hb\in\r$ we have $\epsmo\rot\Hb=\zet\eb+\J$
yielding
\begin{align*}
\cM_{+,\rot}(\Hti,\et;\Hb)
&=\normmuom{\Hb-\Hti}
+\chmom\normepsoms{\eb-\et}\\
&\leq\normmumoom{\rot(\Eb-\Et)}
+\chmom\kamo\big(\normepsoms{\zets(\Eb-\Et)}
+\normepsoms{\na(\vb-\vt)}\big).
\end{align*}
Again, for $\Hh\in\r$ we get
$\cMt_{+,\rot}(\Et,\et;\Hb-\Hh)
=\cM_{+,\rot}(\Hti,\et;\Hb)$.

A main consequence from the third and the last estimates 
in the above lemma is the following a posteriori error estimate result:

\begin{theo}
\label{theoestinum}
Let $\Et\in\rc$ and $\vt\in\hooms$. Furthermore, 
let $\et:=\eh-\kamo(\zets\Et+\na\vt)\in\Ltepsoms$ and $\Hti:=\mumo\rot\Et+\Hh\in\mumo\dcz$. Then
\begin{align*}
\tnorm{(\bar{H}- \Hti,\bar{j}- \et)}
&=\big(\normmuom{\Hb-\Hti}^2
+\ka\normepsoms{\eb-\et}^2\big)^{1/2}\\
&\leq\cM_{+,\rot}(\Hti,\et;\Psi)
+(\kappa^{-1}\chmom+\kappa^{-1/2})\cM_{+,\pioms}(\Et,\vt;\Upsilon)
\end{align*}
holds for all $\Psi\in\r$ and all $\Upsilon\in\epsmo\dcoms$.
\end{theo}

\begin{rem}
\label{apostremupperboundIIsharp}
In Lemma \ref{apostlemupperboundII} and Theorem \ref{theoestinum} 
the upper bounds are equivalent to the respective norms of the error. More precisely it holds
\begin{align*}
\tnorm{(\Hb-\Hti,\eb-\et)}
&\leq\inf_{\Psi\in\r}\cM_{+,\rot}(\Hti,\et;\Psi)
+(\kamo\chmom+\ka^{-1/2})\inf_{\Upsilon\in\epsmo\dcoms}\cM_{+,\pioms}(\Et,\vt;\Upsilon)\\
&\leq\cM_{+,\rot}(\Hti,\et;\Hb)
+(\kamo\chmom+\ka^{-1/2})\cM_{+,\pioms}(\Et,\vt;\pioms\zets\Eb)\\
&\leq\normmuom{\Hb-\Hti}+(\chmom+2^{1/2}\ck\ka^{1/2})\normepsoms{\eb-\et}\\
&\leq\normmuom{\Hb-\Hti}+3\ck\ka^{1/2}\normepsoms{\eb-\et}\\
&\leq(1+9\ck^2)^{1/2}\tnorm{(\Hb-\Hti,\eb-\et)}.
\end{align*}
Moreover, there exists a constant $c>0$, which can be explicitly estimated as well, such that
\begin{align*}
&\qquad c^{-1}\big(\normmuom{\Hb-\Hti}^2+\normepsom{\Eb-\Et}^2+\normepsoms{\na(\vb-\vt)}^2\big)\\
&\leq\inf_{\Psi\in\r}\cM_{+,\rot}^2(\Hti,\et;\Psi)
+\inf_{\Phi\in\epsmo\d}\cM^2_{+,\div}(\Et;\Phi)
+\inf_{\Upsilon\in\epsmo\dcoms}\cM^2_{+,\pioms}(\Et,\vt;\Upsilon)\\
&\leq c\big(\normmuom{\Hb-\Hti}^2+\normepsom{\Eb-\Et}^2+\normepsoms{\na(\vb-\vt)}^2\big).
\end{align*}
If $\Hh\in\r$, the majorant $\ds\inf_{\Psi\in\r}\cM_{+,\rot}(\Hti,\et;\Psi)$
can be replaced by $\ds\inf_{\Psi\in\r}\cMt_{+,\rot}(\Et,\et;\Psi)$ and the term
$\cM_{+,\rot}(\Hti,\et;\Hb)$ by $\cMt_{+,\rot}(\Et,\et;\Hb-\Hh)$.
\end{rem}

By the latter lemma we have fully computable upper bounds for the terms
$$\normepsoms{\eb-\et},\qquad
\normmumooms{\rot(\eb-\et)},\qquad
\normepsoms{\pioms\zets(\Eb-\Et)}$$
and
$$\normepsom{\Eb-\Et}\leq\dnorm{\Eb-\Et},\qquad
\normmumoom{\rot(\Eb-\Et)}\leq\dnormrot{\Eb-\Et},$$
i.e., for the terms
$$\norm{\eb-\et}_{\roms},\qquad
\norm{\Eb-\Et}_{\r}\leq\dnorm{\Eb-\Et}_{\r},\qquad
\normepsoms{\pioms\zets(\Eb-\Et)}.$$

\subsection{Lower Bounds}

To get a lower bound, we use the simple relation in a Hilbert space
$$\forall\,x\qquad
\norm{x}^2
=\max_{y}\big(2\scp{x}{y}-\norm{y}^2\big)
=\max_{y}\scp{2x-y}{y}.$$
Note that the maximum is attained at $y=x$.
Looking at
$$\tnorm{(\Hb-\Hti,\eb-\et)}^2
=\normmuom{\Hb-\Hti}^2
+\ka\normepsoms{\eb-\et}^2
=\normmumoom{\rot(\Eb-\Et)}^2
+\ka\normepsoms{\eb-\et}^2$$
we obtain with $H:=\rot\Phi$ and $\e:=\zets\Phi$ for some $\Phi\in\rc$
by \eqref{vareqEneb3pde1}
\begin{align*}
&\qquad\tnorm{(\Hb-\Hti,\eb-\et)}^2\\
&=\normmumoom{\rot(\Eb-\Et)}^2
+\kamo\normepsoms{\pioms\zets\Eb-\zets\Et-\na\vt}^2\\
&=\max_{H\in\Lt}\scpmumoom{2\rot(\Eb-\Et)-H}{H}
+\kamo\max_{\e\in\Ltoms}\scpepsoms{2(\pioms\zets\Eb-\zets\Et-\na\vt)-\e}{\e}\\
&\geq\scpmumoom{2\rot\Eb-\rot(2\Et+\Phi)}{\rot\Phi}
+\kamo\scpepsoms{2(\pioms\zets\Eb-\zets\Et-\na\vt)-\zets\Phi}{\zets\Phi}\\
&=\scpepsoms{2(\eh-\kamo\na\vt)-\kamo\zets(2\Et+\Phi)}{\zets\Phi}
+2\scpepsom{\J}{\Phi}
-\scpmumoom{2\mu\Hh+\rot(2\Et+\Phi)}{\rot\Phi}\\
&=\scpepsom{2(\zet\eh+\J-\kamo\zet\na\vt)-\kamo\zet\zets(2\Et+\Phi)}{\Phi}
-\scpmumoom{2\mu\Hh+\rot(2\Et+\Phi)}{\rot\Phi}\\
&=\scpepsom{2(\zet\et+\J)-\kamo\zet\zets\Phi}{\Phi}
-\scpom{2\Hti+\mumo\rot\Phi}{\rot\Phi}\\
&=:\cM_{-}(\Hti,\et;\Phi).
\end{align*}
The maxima are attained at
$\hat{H}:=\rot(\Eb-\Et)$ and $\hat{\e}:=\pioms\zets\Eb-\zets\Et-\na\vt$.
We conclude that the lower bound is sharp.
For this, let $\check{\vb}$, $\check{\vt}\in\ho$ be $\ho$-extensions to $\om$
of $\vb$, $\vt$. Note that Calderon's extension theorem holds since 
$\oms$ is Lipschitz. With a cut-off function $\chi\in\cicom$ satisfying $\chi|_{\oms}=1$ we define
$$\Phi:=\Eb-\Et+\na(\chi(\check{\vb}-\check{\vt}))\in\rc.$$
Then, $\rot\Phi=\rot(\Eb-\Et)=\hat{H}$ and
\begin{align*}
\zets\Phi
&=\zets(\Eb-\Et)+\na\zets(\chi(\check{\vb}-\check{\vt}))
=\zets(\Eb-\Et)+\na\zets(\check{\vb}-\check{\vt})\\
&=\zets(\Eb-\Et)+\na(\vb-\vt)
=\pioms\zets\Eb-\zets\Et-\na\vt
=\hat{\e}.
\end{align*}
Alternatively, we can insert $\e:=\pioms\zets\Phi$ into the second maximum, yielding
\begin{align*}
&\qquad\tnorm{(\Hb-\Hti,\eb-\et)}^2\\
&\geq\scpmumoom{2\rot\Eb-\rot(2\Et+\Phi)}{\rot\Phi}
+\kamo\scpepsoms{2(\pioms\zets\Eb-\zets\Et-\na\vt)-\pioms\zets\Phi}{\pioms\zets\Phi}\\
&=\scpmumoom{2\rot\Eb-\rot(2\Et+\Phi)}{\rot\Phi}
+\kamo\scpepsoms{2\pioms\zets(\Eb-\Et)-\pioms\zets\Phi}{\pioms\zets\Phi}\\
&=\scpepsom{2(\zet\eh+\J)-\kamo\zet\pioms\zets(2\Et+\Phi)}{\Phi}
-\scpmumoom{2\mu\Hh+\rot(2\Et+\Phi)}{\rot\Phi}\\
&=\scpepsom{2(\zet\eh+\J)-\kamo\zet\pioms\zets(2\Et+\Phi)}{\Phi}
-\scpom{2\Hti+\mumo\rot\Phi}{\rot\Phi}\\
&=:\cM_{-,\pioms}(\Et,\Hti;\Phi).
\end{align*}
In general, this lower bound is not sharp.
It is sharp, if and only if $\zets\Et+\na\vt\in R(\pioms)$, 
if and only if $\zets\Et+\na\vt=\pioms\zets\Et$,
since then we can choose $\Phi:=\Eb-\Et$ yielding 
$\rot\Phi=\hat{H}$ and $\pioms\zets\Phi=\hat{\e}$.

\begin{lem}
\label{apostlemlowerbound}
Let $\Et\in\rc$ and $\vt\in\hooms$. Then
$$\tnorm{(\Hb-\Hti,\eb-\et)}^2
=\max_{\Phi\in\rc}\cM_{-}(\Hti,\et;\Phi)
\geq\sup_{\Phi\in\rc}\cM_{-,\pioms}(\Et,\Hti;\Phi).$$
\end{lem}

\subsection{Two-Sided Bounds}

Combining Theorem \ref{theoestinum} and  Lemma \ref{apostlemlowerbound}, we have

\begin{theo}
\label{aposttheotwosided}
Let $\Et\in\rc$ and $\vt\in\hooms$. Then
\begin{align*}
\sup_{\Phi\in\rc}\cM_{-,\pioms}(\Et,\Hti;\Phi)
&\leq\max_{\Phi\in\rc}\cM_{-}(\Hti,\et;\Phi)
=\tnorm{(\Hb-\Hti,\eb-\et)}^2
=\normmuom{\Hb-\Hti}^2
+\ka\normepsoms{\eb-\et}^2\\
&\leq\big(\inf_{\Psi\in\r}\cM_{+,\rot}(\Hti,\et;\Psi)
+(\kamo\chmom+\ka^{-1/2})\inf_{\Upsilon\in\epsmo\dcoms}\cM_{+,\pioms}(\Et,\vt;\Upsilon)\big)^2,
\end{align*}
where
\begin{align*}
\cM_{+,\rot}(\Hti,\et;\Psi)
&=\normmuom{\Hti-\Psi}
+\chmom\normepsom{\zet\et+\J-\epsmo\rot\Psi},\\
\cM_{+,\pioms}(\Et,\vt;\Upsilon)
&=\normepsoms{\zets\Et+\na\vt-\Upsilon}
+\chpoms\normoms{\div\eps\Upsilon},\\
\cM_{-}(\Hti,\et;\Phi)
&=\scpepsom{2(\zet\et+\J)-\kamo\zet\zets\Phi}{\Phi}
-\scpom{2\Hti+\mumo\rot\Phi}{\rot\Phi}.
\intertext{If $\Hh\in\r$, $\cM_{+,\rot}$ can be replaced by $\cMt_{+,\rot}$ with}
\cMt_{+,\rot}(\Et,\et;\Psi)
&=\normmumoom{\rot\Et-\mu\Psi}
+\chmom\normepsom{\zet\et+\J-\epsmo\rot(\Psi+\Hh)}.
\end{align*}
\end{theo}

\section{Adaptive Finite Element Method} 

Based on the a posteriori error estimate proven in Theorem \ref{theoestinum} of the previous section, 
we present now an adaptive finite element method (AFEM) for solving the optimal control problem. 
The method consists of a successive loop of the sequence 
\begin{equation}
\label{Cycle}
\mbox{SOLVE} \rightarrow \mbox{ESTIMATE} \rightarrow \mbox{MARK} \rightarrow \mbox{REFINE} \, .
\end{equation}
For solving the optimal control problem, we employ a mixed finite method 
based on the lowest-order edge elements of N\'{e}d\'{e}lec's first family 
and piecewise linear continuous elements.  
Furthermore, the marking of elements for refinement is carried out 
by means of the D\"orfler marking.

\subsection{Finite Element Approximation}

From now on, $\Omega$ and $\omega$ are additionally assumed to be polyhedral. 
For simplicity we set $\eps:=1$.
Let $(h_n)$ denote a monotonically decreasing sequence 
of positive real numbers and let $ \big(\mathcal{T}_h(\Omega)\big)_{h_n} $  
be a nested shape-regular family of simplicial triangulations of $ \Omega$. 
The nested family is constructed in such a way that $ \mu$ 
is elementwise polynomial on $ \mathcal{T}_{h}(\Omega)$, 
and that there exists a subset $ \mathcal T_{h}(\omega) \subset \mathcal{T}_{h}(\Omega)$ such that
\[
\overline\omega = \bigcup_{T \in \mathcal T_{h}(\omega)} T.
\]
For an element $ T \in \mathcal{T}_h(\Omega)$, 
we denote by $\delta_T$ the diameter of $T$ and set 
$ \delta := \max\big\{h_T \, : \, T \in \mathcal{T}_h(\Omega)\big\}$ for the maximal diameter. 
We  consider the lowest-order edge elements of N\'{e}d\'{e}lec's first family  
\[
 \mathcal{N}_1(T):= \big\{ \Phi : T \rightarrow \mathbb{R}^3 \, : \,  \Phi(x) =  a +  b \times x  \textnormal{ with }  a, b \in \mathbb R^3\big\} ,
\]
which give rise to the $\rot$-conforming  N\'{e}d\'{e}lec edge element space   \cite{NED80}
\[
 \overset\circ{\mathsf{R}}{}_h := \big\{ \Phi_h \in \rc(\Omega) \, : \, \Phi_h|_T \in \mathcal{N}_1(T)\quad\forall\,T  \, \in \mathcal{T}_h(\Omega) \big\} .
\]
 Furthermore, we denote the space of piecewise linear continuous elements by
\[
\overset\circ{\mathsf{H}}{}^1_h:= \big\{ \varphi_h \in \hoc(\Omega) \, : \, \varphi_h|_T(x)= a_T + b_T \cdot x  \textnormal{ with }  a_T \in \mathbb R, \, b_T \in \mathbb R^3 \quad \forall\,T  \, \in \mathcal{T}_h(\Omega) \big\} 
\]
and 
\[
\mathsf{H}^1_{\omega,h}:= \big\{ \phi_h \in \ho(\omega) \, : \, \phi_h|_T(x)= a_T + b_T \cdot x  \textnormal{ with }  a_T \in \mathbb R, \, b_T \in \mathbb R^3 \quad \forall\,T  \, \in \mathcal{T}_h(\omega) \big\} .
\]
We formulate now the mixed finite element approximation of  the necessary and sufficient optimality condition \eqref{vareqEneb7pde1}-\eqref{vareqEneb7pde3}, see also \eqref{vareqEneb9pde1}-\eqref{vareqEneb9pde3}
resp. \eqref{vareqEneb10pde}, as follows:
Find $(\bar E_h,\bar u_h,\bar v_h)\in \overset\circ{\mathsf{R}}{}_h \times \overset\circ{\mathsf{H}}{}^1_h \times \mathsf{H}^1_{\omega,h}$  such that, for all $(\Phi_h,\varphi_h,\phi_h) \in \overset\circ{\mathsf{R}}{}_h \times \overset\circ{\mathsf{H}}{}^1_h\times \mathsf{H}^1_{\omega,h}$, there holds
\begin{align}
\label{vareqEh1}
\tilde{a}(\bar E_h,\Phi_h)+b(\Phi_h,\bar u_h)+c(\Phi_h,\bar v_h)&=f(\Phi_h),\\
\label{vareqEh2}
b(\bar E_h,\varphi_h)&=0,\\
\label{vareqEh3}
c(\bar E_h,\phi_h)+d(\bar v_h,\phi_h)&=0,
\end{align}
where
\[
 \tilde{a}(\bar E_h,\Phi_h) = 
\scpmumoom{\rot \bar E_h}{\rot\Phi_h}
+\kamo\scpoms{\zets \bar E_h}{\zets\Phi_h},
\]
and 
\[
b(\Phi_h,\bar u_h) = 
\scpom{\Phi_h}{\na \bar u_h},
\quad
c(\Phi_h,\bar v_h)=  \kamo\scpoms{\zets\Phi_h}{\na \bar v_h},
\quad
d(\bar v_h,\phi_h)=\kappa^{-1} \scpoms{\na \bar v_h}{\na\phi_h}.
\]
As in the continuous case (see Remark \ref{dspremproof}), the existence of a unique solution $(\bar E_h,  \bar v_h,\bar   v_h)\in \overset\circ{\mathsf{R}}{}_h \times \overset\circ{\mathsf{H}}{}^1_h \times \mathsf{H}^1_{\omega,h}$ for the discrete system \eqref{vareqEh1}-\eqref{vareqEh3} follows  from the discrete  Ladyzhenskaya-Babu\v{s}ka-Brezzi  condition:
\begin{equation}\label{LBBh}
\inf_{0 \neq\varphi_h \in\overset\circ{\mathsf{H}}{}^1_h}
\sup_{(\Phi_h,\phi_h)\in \overset\circ{\mathsf{R}}{}_h \times \mathsf{H}^1_{\omega,h} }
\frac{b \big(\Phi_h ,\varphi_h \big)}{\norm{(\Phi_h,\phi_h)}_{\r\times\hoboms}\norm{\varphi_h}_{\hoc}}
\geq1,
\end{equation}
which is obtained, analogously to the continuous case, by setting $\Phi_h = \nabla \varphi_h$ and $\phi_h=0$. Note that the inclusion    $\nabla \overset\circ{\mathsf{H}}{}^1_h \subset \overset\circ{\mathsf{R}}{}_h$ holds such that every gradient field $\nabla \varphi_h$ of a piecewise linear continuous function $\varphi_h \in \overset\circ{\mathsf{H}}{}^1_h$ is an element of $\overset\circ{\mathsf{R}}{}_h$.   Let us also remark that
on the discrete solenoidal subspace of $\overset\circ{\mathsf{R}}{}_h$ 
the following discrete Maxwell estimate holds:
$$\exists\,c>0
\quad 
\forall\, \Phi_h \in 
\big\{\Psi_h \in \overset\circ{\mathsf{R}}{}_h \,Ê: \,
\scpom{\Psi_h}{\na   \psi_h} = 0 \quad
\forall\, \psi_h \in \overset\circ{\mathsf{H}}{}^1_h\big\}
\qquad
\normom{\Phi_h}  \leq  c \, \normom{ \rot \Phi_h}.$$
Note that $c$ is independent of $h$, see e.g. \cite{hip02}.
Having solved the discrete system \eqref{vareqEh1}-\eqref{vareqEh3}, 
we obtain the finite element approximations for the optimal control 
and the optimal magnetic field as follows 
\begin{equation}
\bar{j}_h :=  j_{d,h}-\kamo(  \bar E_h \vert_\omega +\na \bar   v_h) \qquad
\bar{H}_h := \mumo\rot \bar E_h + H_{d,h},
\end{equation}
see \eqref{defforj}, where $ j_{d,h}$ and $H_{d,h}$ are appropriate finite element approximations of  the shift control $j_d$ and the desired magnetic field $H_d$, respectively.  

\subsection{Evaluation of the Error Estimator}

By virtue of Theorem \ref{theoestinum},  
the total error in the finite element solution can be estimated by
\begin{equation}
\label{estinum}
\tnorm{(\bar{H}-\bar{H}_h,\bar{j}-\bar{j}_h)}   
\leq \cM_{+,\rot}(\bar{H}_h,\bar{j}_h;\Psi)
+(\kappa^{-1}\chmom+\kappa^{-1/2})\cM_{+,\pioms}(\bar E_h,\bar v_h;\Upsilon),
\end{equation}
for every $(\Psi, \Upsilon) \in  R(\Omega) \times \dcoms $,  where 
\begin{align}\label{definumM1}
\cM_{+,\rot}(\bar{H}_h,\bar{j}_h ;\Psi)
&=\normmuom{\bar{H}_h-\Psi}
+\chmom\normom{\zet\bar{j}_h+\J-\rot\Psi},\\
\label{definumM2}
\cM_{+,\pioms}(\bar E_h,\bar   v_h;\Upsilon)
&\;=\normoms{\zets \bar E_h+\na \bar   v_h-\Upsilon}
+\chpoms\normoms{\div \Upsilon}.
\end{align}
We point out that $(\Psi, \Upsilon) \in  R(\Omega) \times \dcoms$ 
should be suitably chosen in order to avoid big over estimation in \eqref{estinum}.  Our strategy is to find  appropriate  finite element functions for $\Psi$ and $\Upsilon$, which minimize  functionals related to   $\cM_{+,\rot}$ and $\cM_{+,\pioms}$.     To this aim, we make use of the $\rot$-conforming  N\'{e}d\'{e}lec edge element space without the vanishing tangential trace condition 
\[
\mathsf{R}_h := \big\{ \Psi_h \in \r(\Omega) \, : \, \Psi_h|_T \in \mathcal{N}_1(T) \quad \forall\,T  \, \in \mathcal{T}_h(\Omega) \big\} 
\]
and  the   $\div$-conforming Raviart-Thomas finite element space  on the control domain
\[
 \overset\circ{\mathsf{D}}{}_{\omega,h} := \big\{ \Upsilon_h \in  \dcoms  \, : \, \Upsilon_h|_T \in \mathcal{RT}_1(T) \quad \forall\,T  \, \in \mathcal{T}_h(\omega) \big\},
\]
where  
 \[
 \mathcal{RT}_1(T):= \{ \Upsilon : T \rightarrow \mathbb{R}^3 \, : \,  \Upsilon(x) =  a +  b x \textrm{ with } a \in \mathbb R^3, b \in \mathbb R\}.
\]
Now, we look for solutions of the finite-dimensional minimization problems
\begin{equation}
\label{optinum1}
\min_{\Psi_h \in \mathsf{R}_h} 
\Big( \normmuom{\bar{H}_h-\Psi_h}^2
+\chmom^2 \normom{\zet\bar{j}_h+\J- \rot\Psi_h}^2\Big)
\end{equation}
and
\begin{equation}
\label{optinum2}
\min_{ \Upsilon_h \in \overset\circ{\mathsf{D}}{}_{\omega,h}} 
\Big(\normoms{\zets \bar E_h+\na \bar   v_h-\Upsilon_h}^2
+\chpoms^2\normoms{\div \Upsilon_h}^2\Big).
\end{equation}
Evidently, the optimization problems \eqref{optinum1}-\eqref{optinum2}   admit unique solutions $\bar \Psi_h \in \mathsf{R}_h$ and $\bar \Upsilon_h \in \overset\circ{\mathsf{D}}{}_{\omega,h}$. Furthermore, the corresponding necessary and sufficient optimality conditions are given by the coercive variational equalities
\begin{align*}
\forall\,\Psi_h &\in \mathsf{R}_h 
&
\chmom^2\scpom{\rot  \bar \Psi_{h}}{\rot \Psi_h}
+ \scpmuom{ \bar \Psi_{h}}{\Psi_h}
&= \scpmuom{ \bar{H}_h}{\Psi_h}
+\chmom^2\scpom{\zet\bar{j}_h+\J}{\rot\Psi_h}\\
\forall\,\Upsilon_h &\in \overset\circ{\mathsf{D}}{}_{\omega,h} 
&
\chpoms^2\scpoms{\div \bar \Upsilon_h}{\div \Upsilon_h}
+\scpoms{{\bar \Upsilon_h}}{\Upsilon_h}
&=\scpoms{\zets \bar E_h +\na \bar   v_h}{\Upsilon_h}.
\end{align*}
Taking the optimal solutions of \eqref{optinum1}-\eqref{optinum2} into account, we introduce  
\begin{equation}
\label{DefMh}
\mathcal M_h :=
 \cM_{+,\rot}(\bar{H}_h,\bar{j}_h ; \bar \Psi_h) +  
 (\kappa^{-1} \chmom + \kappa^{-1/2}) \cM_{+,\pioms}(\bar E_h,\bar v_h; \bar \Upsilon_h).
\end{equation}
Then, \eqref{estinum} yields 
\begin{equation} \label{estinumtheo}
\tnorm{(\bar{H}-\bar{H}_h,\bar{j}-\bar{j}_h)} \le \mathcal M_h.
\end{equation}

\subsection{D\"orfler Marking}

In the step MARK of the sequence \eqref{Cycle}, elements of the simplicial triangulation $ \mathcal{T}_h(\Omega)$ are marked for refinement according to the information provided by    the estimator $\mathcal M_h$. 
With regard to convergence and quasi-optimality of AFEMs, the bulk criterion by D\"orfler \cite{doe96} is a reasonable choice for the marking strategy, which we pursue here.  More precisely, we select a set $\mathcal{E}$ of elements such that for some $  \theta \in (0,1) $ there holds
\begin{align}
 \label{MARK2}
  \sum\limits_{T \in \mathcal{E}}  \mathcal M_T  \ge \theta   \sum\limits_{T \in \mathcal{T}_h(\Omega)}  \mathcal M_T  ,
\end{align}
where 
\[
\begin{split}
  \mathcal M_T := & |\bar{H}_h- \bar \Psi_h|_{T,\mu}
+\chmom | \zet\bar{j}_h+\J-\epsmo\rot \bar \Psi_h|_{T}  +    \big(\kappa^{-1} \chmom + \kappa^{-1/2}\big)  \mathcal M_{\omega,T}  \\[2ex]
   \mathcal M_{\omega,T} := & \left\{ 
   \begin{array}{lll} 
  |\zets \bar E_h+\na \bar   v_h- \bar \Upsilon_h|_{T}
+\chpoms|\div  \bar \Upsilon_h|_{T}  \quad &\textrm{if } T \in \mathcal{T}_h(\omega), \\
0 & \textrm{if } T \notin \mathcal{T}_h(\omega).
\end{array}
\right.
\end{split}
\]
Elements of the triangulation $ \mathcal{T}_h(\Omega)$ that have been marked for refinement are subdivided by the newest vertex bisection.

\subsection{Analytical Solution}

To test the numerical performance of the previously introduced adaptive method, we construct an analytical solution for the optimal control problem \eqref{optcontprobintro}. 
Here, the computational domain and the control domain are specified by 
$$
\Omega :=(-0.5,1)^3 \quad \textrm{and} \quad \omega := (0,0.5)^3.
$$
Furthermore, we put $\eps:=1$, $\kappa:=1$, 
and the  magnetic permeability is set to be piecewise constant, i.e.
$$
\mu := \left\{ 
\begin{array}{lll}
10 \quad &\textrm{in } (-0.5,0) \times (-0.5,0) \times (-0.5,1),\\
1 &\textrm{elsewhere}.\\
\end{array}
\right.
$$
We introduce  the vector field
$$
E(x) :=  \frac{\mu^2(x)}{8 \pi^2}\sin^2(2\pi x_1) \sin^2(2\pi x_2)\dvec001 \quad \forall\, x \in \Omega,
$$
and set
$$
\bar E :=  \chi_{_{\Omega_s}} E \quad \textrm{and} \quad \bar{H} := \mu^{-1} \rot E,
$$
where $\chi_{_{\Omega_s}}$ stands for the characteristic function on the subset 
$\Omega_s :=  \Omega \setminus \big\{(0,0.5) \times (0,0.5) \times (-0.5,1)\big\}$. 
By construction,  it holds that $\bar E \in \rc(\Omega) \cap \dzom$ and $\bar{H} \in  \r(\Omega) \cap \mumo\dczom$.   The desired magnetic field is set to be
$$
\Hh := \chi_{_{\Omega \setminus \Omega_s}} \bar{H} \in \r(\Omega).
$$
Finally, we define the optimal control $\bar{j} \in \dcz(\omega)$ as
$$ 
\bar{j} (x) := 100\dvec{\sin(2\pi x_1) \cos(2 \pi x_2)}{-\sin(2 \pi x_2) \cos(2\pi x_1)}{0} 
\quad \forall\, x \in \omega,
$$
and the shift control $\eh$ as well as the applied electric current $J$ as
$$
\eh := \bar{j} \quad \textrm{and} \quad J :=  \left\{ 
\begin{array}{lll} \rot \bar{H} - \bar{j} \quad &\textrm{in } \omega,  \\[1ex]
\rot  \bar{H} &\textrm{elsewhere}.\\
\end{array}
\right.
$$
By construction, we have 
\begin{align*}
\hspace{3cm} \rot\bar{H}&= \zet\eb+ \J,&\rot\bar E&=\mu(\bar{H}-\Hh)&\text{in }\om,\\
\div\mu\bar{H}&=0,&\div\bar E&=0&\text{in }\om,\\
n\cdot\mu\bar{H}&=0,&n\times \bar E&=0&\text{on }\ga,
\end{align*}
and 
\[
 \dcz(\omega) \ni \eb=\eh= \eh -\frac{1}{\ka} \pi_\omega\zets\Eb, 
 \]
from which it follows that  $\bar{j}$ is the optimal control of \eqref{optcontprobintro} 
with the associated optimal magnetic field $\bar{H} $ and the adjoint field $\bar E$.

\subsection{Numerical Results}

With the constructed analytical solution  at hand, we can now demonstrate  the numerical performance of the adaptive method using the proposed error estimator $\mathcal M_h$ defined in  \eqref{DefMh}. Here, we used a moderate value $\theta =0.5$  for the bulk criterion in the D\"orfler marking.  Let us also point out that   all numerical results  were implemented by a Python script using the Dolphin Finite Element Library \cite{fenics12}. In the first experiment, we carried out a  thorough comparison between the total error 
$
\tnorm{(\bar{H}-\bar{H}_h,\bar{j}-\bar{j}_h)}
$
resulting from the adaptive mesh refinement strategy and the one based on  the uniform mesh refinement. The result is plotted in  Figure \ref{fig1}, where DoF stands for the degrees of freedom in the finite element space. Based on this result,   we   conclude a better convergence performance of the adaptive method over the standard uniform mesh refinement.  Next,  in Table \ref{tab1}, we  report on  the detailed convergence history for the total error including the  value for $\mathcal M_h$ computed in every step of the adaptive mesh refinement method. It should be underlined that the Maxwell and  Poincar\'e constants $ \chmom $ and $\chpoms$  appear in the proposed estimator  $\mathcal M_h$ (see \eqref{definumM1}-\eqref{definumM2} and \eqref{DefMh}).    We do not neglect these constants in our computation, and    there is no  further unknown or hidden constant in $\mathcal M_h$. By the choice of the magnetic permeability $\mu$ and the computational domains $\Omega,\omega$ (see  Remark \ref{apostremupperboundI}), the constants $ \chmom,\chpoms$ can be estimated as follows:
$$\chmom \leq  15 \frac{\sqrt{3} }{\pi}  \qquad \textrm{and} \quad
\chpoms \le   \frac{\sqrt{3} }{ 2\pi}
$$ 
These values were used in the computation of $\mathcal M_h$. As we can observe in Table 1, $\mathcal M_h$ severs as an upper bound for the total error. This is in accordance with our theoretical findings.  

\begin{figure}[h!]
\centering
\includegraphics[scale=0.3]{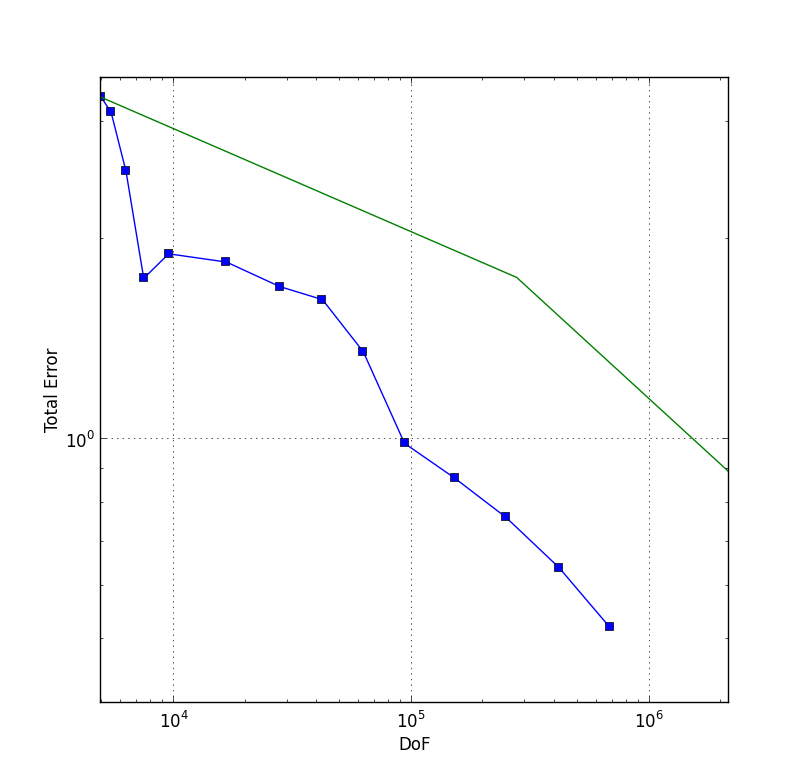}
\caption{Total error for uniform (green line) and adaptive mesh  refinement (blue line).}
\label{fig1}
\end{figure}

\begin{table}[!ht]
\begin{center}
\setlength{\abovecaptionskip}{1pt}
\setlength{\belowcaptionskip}{1pt}
\renewcommand{\arraystretch}{0.7}
\setlength{\tabcolsep}{0.14cm}
\begin{tabular*}{0.6\textheight}
{@{\extracolsep{\fill}}cccccccccc}
\hline
DoF & Error in $H$&  Error in $j$ & Total Error  & $\mathcal M_h$\\
\hline
\hline
$4940$&$0.864259760285$&$ 3.15539577688$&$3.2716154178$&$ 63.4376616999$\\
$ 5436$&$0.694612463498$&$3.02692021715$&$3.10559695959$&$58.5220353976$\\
$ 6280$&$0.560747440261$&$2.46658970377$&$2.52952613319$&$46.1596277893$\\
$ 7480$&$0.517270941002$&$1.66980235746$&$1.74808728025$&$ 29.9835458365$\\
$9506 $&$ 0.486958908788$&$1.83890409144$&$1.90228736955$&$33.7781950898$\\
$16593$&$0.409942119878$&$1.79996131396$&$1.8460534319$&$27.7781692767$\\
$27622$&$0.322357401619$&$1.66560722229$&$1.69651457799$&$22.1793926139$\\
$42000 $&$0.284583422125$&$1.59619732314$&$1.62136782334$&$20.1292192945$\\
$62424 $&$ 0.234023588085$&$1.33186688758$&$1.35227084788$&$16.7472327351$\\
$ 92730 $&$0.196145507066$&$0.963057265783$&$ 0.982828752692 $&$12.4090773249$\\
$150802$&$0.166713389106$&$ 0.857068785338$&$ 0.873132439501$&$10.621022309$\\
$ 248269 $&$ 0.143328090061$&$0.747991599295$&$ 0.761599877899$&$9.09719391479$\\
$ 414395$&$0.120042829228$&$ 0.630681094598$&$0.642003834827$&$7.62309929568$\\
$ 674856 $&$0.102521829252$&$0.510228751611$&$ 0.520426848311$&$ 6.30611525921$\\
\hline
\end{tabular*}
\end{center}
\caption{Convergence history.}
\label{tab1}
\end{table}

In Figure \ref{fig2}, we plot the  finest   mesh as the result of the adaptive method. It is noticeable that  the adaptive mesh refinement  is mainly concentrated  in the control domain. Moreover, the computed optimal control and optimal magnetic field are depicted in Figure \ref{fig3}. We see that they are already close to the optimal one.

\begin{figure}[h!]
\centering
\includegraphics[width=5cm,height=5cm]{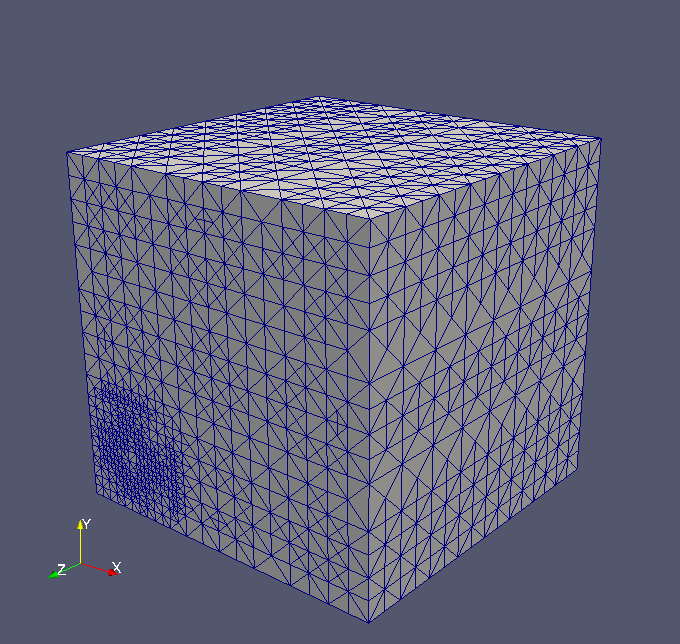}
\includegraphics[width=5cm,height=5cm]{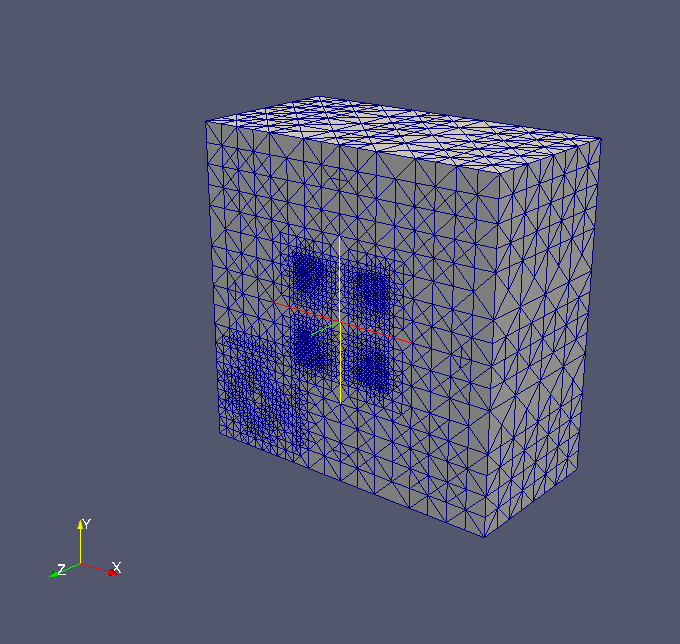}
\caption{Adaptive mesh.}
\label{fig2}
\end{figure}

\begin{figure}[h!]
\centering
\includegraphics[width=5cm,height=5cm]{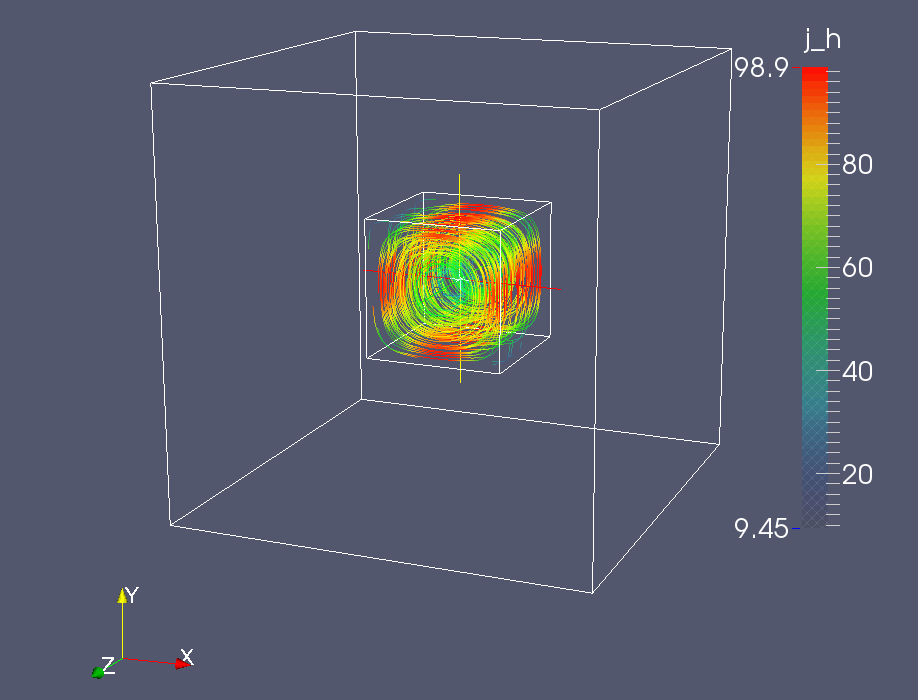}
\includegraphics[width=5cm,height=5cm]{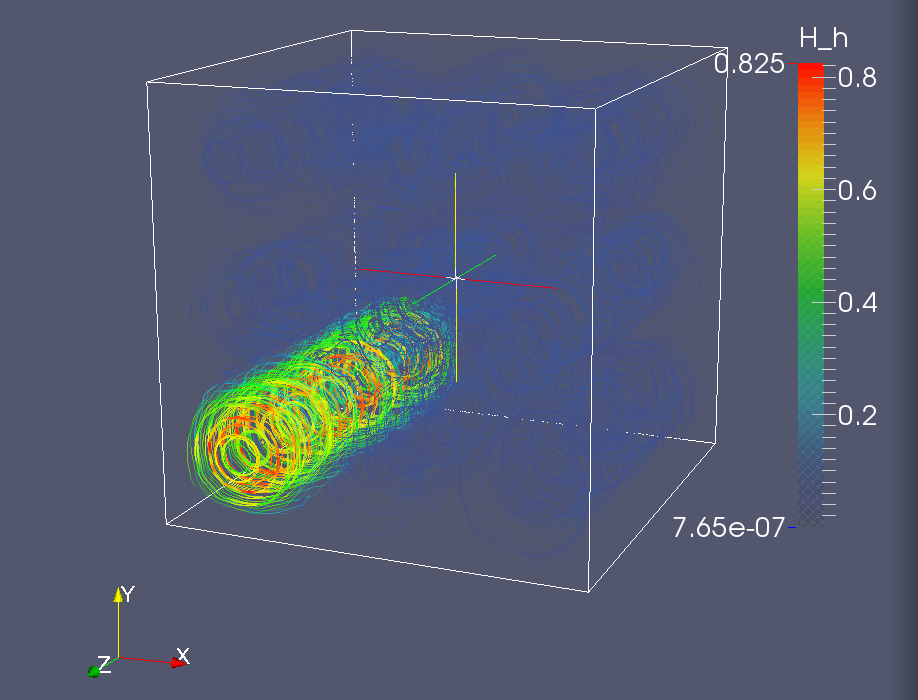}
\caption{Computed optimal control (left plot) and optimal magnetic field (right plot) on the finest adaptive mesh.}
\label{fig3}
\end{figure}
 
In our second test, we carried out a numerical  experiment by making use of   the exact total error 
$
\tnorm{(\bar{H}-\bar{H}_h,\bar{j}-\bar{j}_h)}
$
as the estimator (exact estimator) in the adaptive mesh refinement. More precisely, we replaced $\mathcal M_T$ in the D\"orfler marking strategy \eqref{MARK2} by the exact total error over each element $ T \in \mathcal{T}_h(\Omega)$. Figure \ref{fig4} depicts  the computed total error resulting from this adaptive technique compared with our method. Here, the convergence performance of the mesh refinement strategy using the exact estimator turns out to be quite similar to the one based on the estimator $\mathcal M_h$. Also, the resulting adaptive meshes from these two methods exhibit a similar structure, see Figure \ref{fig5}.  
 Based on these numerical results, 
we finally conclude that the  proposed a posteriori estimator $\mathcal M_h$  is indeed suitable for an adaptive mesh refinement strategy, in order to improve the convergence performance  of the finite element solution towards the optimal one.

\begin{table}[!ht]
\begin{center}
\setlength{\abovecaptionskip}{1pt}
\setlength{\belowcaptionskip}{1pt}
\renewcommand{\arraystretch}{0.8}
\setlength{\tabcolsep}{0.2cm}
\begin{tabular*}{0.6\textheight}
{@{\extracolsep{\fill}}cccccccccc}
\hline
DoF & Error in $H$&  Error in $j$ & Total Error  \\
\hline
\hline         
$4940$&$0.864259760285$&$ 3.15539577688$&$3.2716154178$ \\
$ 5372$&$0.700582925336$&$3.0269236357$&$3.10694112137$ \\
$ 5956$&$0.567880369596$&$2.59095417982$&$2.65245766717$ \\
$ 6866$&$0.525899386428$&$1.65477728914$&$1.73633465706$ \\
$7975 $&$ 0.491051451195$&$1.79991321699$&$1.86569534395$ \\
$13420$&$0.475834638164$&$1.68710457122$&$1.75292339739$ \\
$21122$&$0.469036197488$&$1.76583157736$&$1.82706215389$ \\
$31404 $&$0.459163475711$&$1.65610319012$&$1.71857757281$ \\
$44722 $&$ 0.438814299362$&$1.41717667783$&$1.48355914123$ \\
$ 62092 $&$0.377265302988$&$1.09347162408$&$ 1.15672351991 $ \\
$88972$&$0.297757792322$&$ 0.883606131143$&$ 0.932426671584$ \\
$ 129694 $&$ 0.268987264855$&$0.837765084641$&$ 0.879888905316$ \\
$ 215804$&$0.208852836651$&$ 0.721694386498$&$0.751307057654$ \\
$ 334072 $&$0.194097809391$&$0.587416582193$&$ 0.618653538457$ \\
$ 538189 $&$0.157893445276$&$0.494322025147$&$ 0.518926396136$ \\
\hline
\end{tabular*}
\end{center}
\caption{Convergence history for the adaptive  refinement  using  the exact estimator.}
\label{tab2}
\end{table}

\begin{figure}[h!]
\centering
\includegraphics[scale=0.3]{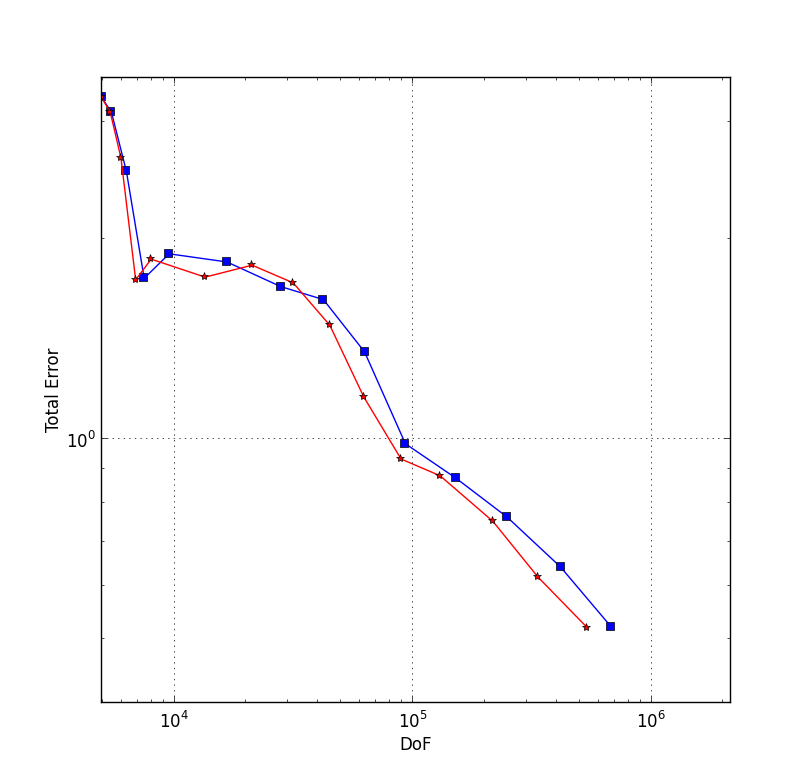}
\caption{Total error for the  adaptive  refinement strategies based on the exact estimator (red line) and the estimator  $\mathcal M_h$ (blue line).}
\label{fig4}
\end{figure}

\begin{figure}[h!]
\centering
\includegraphics[width=5cm,height=5cm]{mesh_adaptive1.png}
\includegraphics[width=5cm,height=5cm]{mesh_adaptive2.png}\\
\includegraphics[width=5cm,height=5cm]{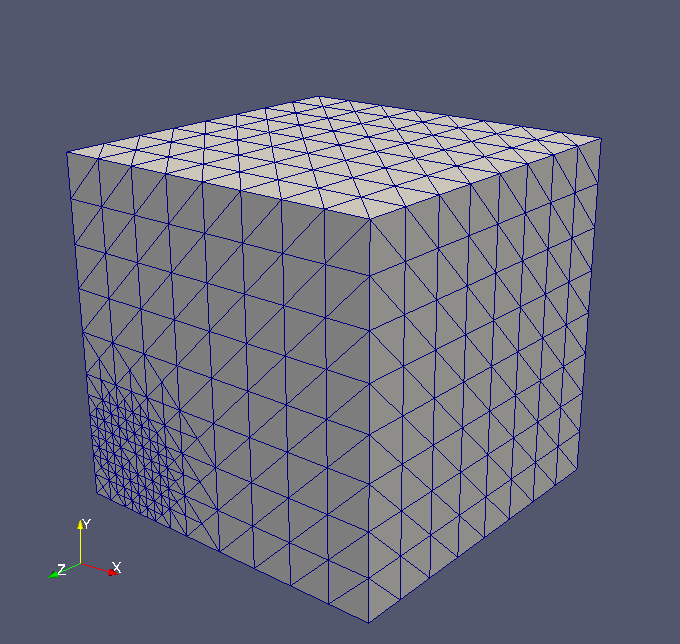}
\includegraphics[width=5cm,height=5cm]{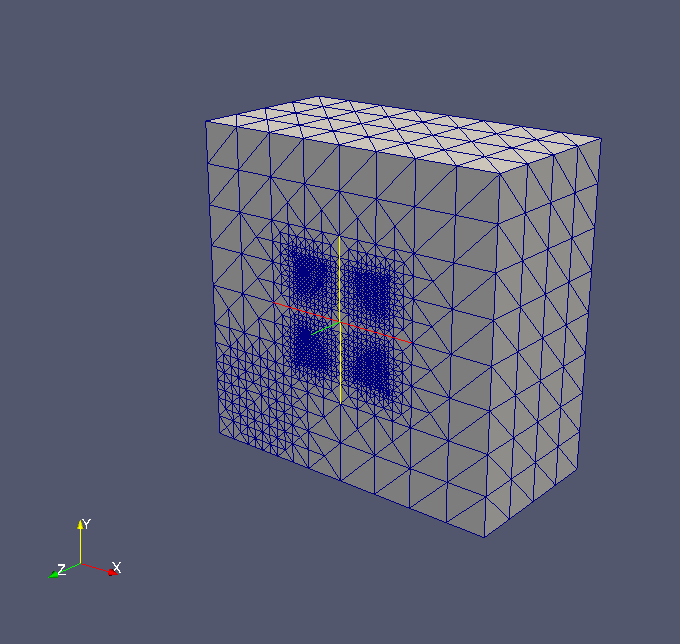}
\caption{Adaptive mesh resulting from the estimator $\mathcal M_h$ (upper plot) and the exact estimator (lower plot).}
\label{fig5}
\end{figure}

{\footnotesize
\bibliographystyle{plain} 
\bibliography{paule-irwin}

\begin{thebibliography}{10}

\bibitem{bauerpaulyschomburgmcpweaklip}
S.~Bauer, D.~Pauly, and M.~Schomburg.
\newblock The {M}axwell compactness property in bounded weak {L}ipschitz
  domains with mixed boundary conditions.
\newblock {\em SIAM J. Math. Anal.}, 2016.

\bibitem{bebendorfpoincareconvex}
M.~Bebendorf.
\newblock A note on the {P}oincar\'e inequality for convex domains.
\newblock {\em Z. Anal. Anwendungen}, 22(4):751--756, 2003.

\bibitem{doe96}
W.~D{\"o}rfler.
\newblock A convergent adaptive algorithm for {P}oisson's equation.
\newblock {\em SIAM J. Numer. Anal.}, 33(3):1106--1124, 1996.

\bibitem{giraultraviartbook}
V.~Girault and P.-A. Raviart.
\newblock {\em Finite Element Methods for {N}avier-{S}tokes Equations: Theory
  and Algorithms}.
\newblock Springer (Series in Computational Mathematics), Heidelberg, 1986.

\bibitem{hip02}
R.~Hiptmair.
\newblock Finite elements in computational electromagnetism.
\newblock {\em Acta Numer.}, 11:237--339, 2002.

\bibitem{hopyou14}
R.~H.~W. Hoppe and I.~Yousept.
\newblock Adaptive edge element approximation of { {H}(curl)}-elliptic optimal
  control problems with control constraints.
\newblock {\em BIT}, 55(1):255--277, 2015.

\bibitem{jochmanncompembmaxmixbc}
F.~Jochmann.
\newblock A compactness result for vector fields with divergence and curl in
  ${L}^q({\Omega})$ involving mixed boundary conditions.
\newblock {\em Appl. Anal.}, 66:189--203, 1997.

\bibitem{kollang13a}
M.~Kolmbauer and U.~Langer.
\newblock Efficient solvers for some classes of time-periodic eddy current
  optimal control problems.
\newblock In Oleg~P. Iliev, Svetozar~D. Margenov, Peter~D Minev, Panayot~S.
  Vassilevski, and Ludmil~T Zikatanov, editors, {\em Numerical Solution of
  Partial Differential Equations: Theory, Algorithms, and Their Applications},
  volume~45, pages 203--216. Springer New York, 2013.

\bibitem{kollang13b}
M.~Kolmbauer and U.~Langer.
\newblock A robust preconditioned {M}in{R}es solver for time-periodic eddy
  current problems.
\newblock {\em Comput. Methods Appl. Math.}, 13(1):1--20, 2013.

\bibitem{leisbook}
R.~Leis.
\newblock {\em Initial Boundary Value Problems in Mathematical Physics}.
\newblock Teubner, Stuttgart, 1986.

\bibitem{fenics12}
A.~Logg, K.-A. Mardal, and G.~N. Wells.
\newblock {\em Automated Solution of Differential Equations by the Finite
  Element Method}.
\newblock Springer, Boston, 2012.

\bibitem{NED80}
J.-C. N{\'e}d{\'e}lec.
\newblock Mixed finite elements in {${\bf R}^{3}$}.
\newblock {\em Numer. Math.}, 35(3):315--341, 1980.

\bibitem{NeittaanmakiRepin2004}
P.~Neittaanm{\"a}ki and S.~Repin.
\newblock {\em Reliable methods for computer simulation, error control and a
  posteriori estimates}.
\newblock Elsevier, New York, 2004.

\bibitem{nicstitro14}
S.~Nicaise, S.~Stingelin, and F.~Tr\"oltzsch.
\newblock On two optimal control problems for magnetic fields.
\newblock {\em Computational Methods in Applied Mathematics}, 14(4):555--573,
  2014.

\bibitem{nicstitro15}
S.~Nicaise, S.~Stingelin, and F.~Tr{\"o}ltzsch.
\newblock Optimal control of magnetic fields in flow measurement.
\newblock {\em Discrete Contin. Dyn. Syst. Ser. S}, 8(3):579--605, 2015.

\bibitem{paulymaxconst0}
D.~Pauly.
\newblock On constants in {M}axwell inequalities for bounded and convex
  domains.
\newblock {\em Zapiski POMI{\rm, 435:46-54, 2014}, \& J. Math. Sci. (N.Y.)},
  2014.

\bibitem{paulymaxconst1}
D.~Pauly.
\newblock On {M}axwell's and {P}oincar\'e's constants.
\newblock {\em Discrete Contin. Dyn. Syst. Ser. S}, 8(3):607--618, 2015.

\bibitem{paulymaxconst2}
D.~Pauly.
\newblock On the {M}axwell constants in 3{D}.
\newblock {\em Math. Methods Appl. Sci.}, 2015.

\bibitem{paulyrepinmaxst}
D.~Pauly and S.~Repin.
\newblock Two-sided a posteriori error bounds for electro-magneto static
  problems.
\newblock {\em J. Math. Sci. (N.Y.)}, 166(1):53--62, 2010.

\bibitem{payneweinbergerpoincareconvex}
L.E. Payne and H.F. Weinberger.
\newblock An optimal {P}oincar\'e inequality for convex domains.
\newblock {\em Arch. Rational Mech. Anal.}, 5:286--292, 1960.

\bibitem{picardcomimb}
R.~Picard.
\newblock An elementary proof for a compact imbedding result in generalized
  electromagnetic theory.
\newblock {\em Math. Z.}, 187:151--164, 1984.

\bibitem{picardweckwitschxmas}
R.~Picard, N.~Weck, and K.-J. Witsch.
\newblock Time-harmonic {M}axwell equations in the exterior of perfectly
  conducting, irregular obstacles.
\newblock {\em Analysis (Munich)}, 21:231--263, 2001.

\bibitem{repinbookone}
S.~Repin.
\newblock {\em A posteriori estimates for partial differential equations}.
\newblock Walter de Gruyter (Radon Series Comp. Appl. Math.), Berlin, 2008.

\bibitem{troeval16}
F.~Tr\"oltzsch and A.~Valli.
\newblock Optimal control of low-frequency electromagnetic fields in multiply
  connected conductors.
\newblock To appear in Optimization, DOI:10.1080/02331934.2016.1179301.

\bibitem{troyou12}
F.~Tr{\"o}ltzsch and I.~Yousept.
\newblock P{DE}-constrained optimization of time-dependent 3{D} electromagnetic
  induction heating by alternating voltages.
\newblock {\em ESAIM Math. Model. Numer. Anal.}, 46(4):709--729, 2012.

\bibitem{webercompmax}
C.~Weber.
\newblock A local compactness theorem for {M}axwell's equations.
\newblock {\em Math. Methods Appl. Sci.}, 2:12--25, 1980.

\bibitem{weckmax}
N.~Weck.
\newblock {M}axwell's boundary value problems on {R}iemannian manifolds with
  nonsmooth boundaries.
\newblock {\em J. Math. Anal. Appl.}, 46:410--437, 1974.

\bibitem{witschremmax}
K.-J. Witsch.
\newblock A remark on a compactness result in electromagnetic theory.
\newblock {\em Math. Methods Appl. Sci.}, 16:123--129, 1993.

\bibitem{xuzou16}
Y.~Xu and J.~Zou.
\newblock A convergent adaptive edge element method for an optimal control
  problem in magnetostatics.
\newblock To appear in ESAIM: M2AN, DOI:
  http://dx.doi.org/10.1051/m2an/2016030, 2006.

\bibitem{yosidabook}
K.~Yosida.
\newblock {\em Functional Analysis}.
\newblock Springer, Heidelberg, 1980.

\bibitem{you12b}
I.~Yousept.
\newblock Finite {E}lement {A}nalysis of an {O}ptimal {C}ontrol {P}roblem in
  the {C}oefficients of {T}ime-{H}armonic {E}ddy {C}urrent {E}quations.
\newblock {\em Journal of Optimization Theory and Applications},
  154(3):879--903, 2012.

\bibitem{you12a}
I.~Yousept.
\newblock Optimal control of {M}axwell's equations with regularized state
  constraints.
\newblock {\em Computational Optimization and Applications}, 52(2):559--581,
  2012.

\bibitem{you13}
I.~Yousept.
\newblock Optimal {C}ontrol of {Q}uasilinear {$\boldsymbol
  H(\mathbf{curl})$}-{E}lliptic {P}artial {D}ifferential {E}quations in
  {M}agnetostatic {F}ield {P}roblems.
\newblock {\em SIAM J. Control Optim.}, 51(5):3624--3651, 2013.

\end{thebibliography}
%\bibliography{paule}
%\bibliography{/Users/paule/Library/texmf/tex/TeXinput/bibtex/paule}
}

\end{document}